\documentclass[a4paper,11pt,reqno]{amsart}

\usepackage[utf8]{inputenc}
\usepackage{amsmath}
\usepackage{amssymb}
\usepackage{enumitem}
\usepackage{graphicx}
\usepackage{mathtools}
\usepackage{subcaption}
\usepackage{tikz}
\usepackage[margin=1in]{geometry}
\usepackage[numbers]{natbib}
\usepackage[colorlinks=true,linkcolor=blue]{hyperref}

\allowdisplaybreaks
\newcommand*{\p}{\mathbb{P}}
\newcommand*{\e}{\mathbb{E}}
\newcommand*{\D}{\mathrm{d}}

\newcommand*{\ges}{\geqslant}
\newcommand*{\les}{\leqslant}
\newcommand*{\indf}{1\negmedspace\mathrm{I}}
\newcommand*{\oq}{\overline{q}}
\newcommand*{\ulX}{\underline{X}}

\newcommand*{\oh}{\mathrm{o}}
\newcommand*{\Oh}{\mathrm{O}}
\newcommand*{\mcC}{\mathcal{C}}
\newcommand*{\mcD}{\mathcal{D}}
\newcommand*{\mcF}{\mathcal{F}}
\newcommand*{\mcG}{\mathcal{G}}
\newcommand*{\mcH}{\mathcal{H}}
\newcommand*{\mcS}{\mathcal{S}}
\newcommand*{\mbR}{\mathbb{R}}
\newcommand*{\mbC}{\mathbb{C}}
\newcommand*{\ve}{\varepsilon}

\newcommand*{\cas}{\xrightarrow{\mathrm{a.s.}}}
\newcommand*{\cid}{\xrightarrow{\mathrm{d}}}
\newcommand*{\cip}{\xrightarrow{\mathbb{P}}}
\newcommand*{\var}{\mathrm{var}}
\newcommand*{\cov}{\mathrm{cov}}
\newcommand*{\as}{\mathrm{a.s.}}

\DeclarePairedDelimiter{\abs}{\lvert}{\rvert}

\newtheorem{thm}{Theorem}
\newtheorem{lem}[thm]{Lemma}
\newtheorem{prop}[thm]{Proposition}
\newtheorem{cor}[thm]{Corollary}

\numberwithin{equation}{section}

\begin{document}

\title[Probability of total domination]{Probability of total domination for\\ transient reflecting processes in a quadrant}

\author[V.\ Fomichov]{Vladimir~Fomichov}
\author[S.\ Franceschi]{Sandro~Franceschi}
\author[J.\ Ivanovs]{Jevgenijs~Ivanovs}
\address{Aarhus University and T\'{e}l\'{e}com SudParis}

\begin{abstract}
We consider two-dimensional L\'evy processes reflected to stay in the positive quadrant. Our focus is on the non-standard regime when the mean of the free process is negative but the reflection vectors point away from the origin, so that the reflected process escapes to infinity along one of the axes. Under rather general conditions, it is shown that such behaviour is certain and each component can dominate the other with positive probability for any given starting position. Additionally, we establish the corresponding invariance principle providing justification for the use of the reflected Brownian motion as an approximate model. Focusing on the probability that the first component dominates, we derive a kernel equation for the respective Laplace transform in the starting position. This is done for the compound Poisson model with negative exponential jumps and, by means of approximation, for the Brownian model. Both equations are solved via boundary value problem analysis, which also yields the domination probability when starting at the origin. Finally, certain asymptotic analysis and numerical results are presented.
\end{abstract}


\maketitle

\section{Introduction}
Reflected processes occupy a prominent role in operations research and applied probability literature. In the one-dimensional setting, reflection is specified in terms of the classical Skorokhod problem, and it is widely used to model workload in queues, as well as capital injections and dividends in risk insurance, just to name a few applications. Multidimensional models, allowing for various new features, have been extensively studied as well. We only mention the classical monographs~\cite{cohen_boxma} and~\cite{fayolle_random_2017}, as well as a survey paper~\cite{williams_survey} on the semimartingale reflected Brownian motion. Apart from studying some fundamental properties of the multidimensional model \cite{zwart_SRBM,taylor_williams}, most of the work focuses on the recurrent case and the stationary distribution of the reflected process; see~\cite{dieker_reflected_2009,franceschi_2019} for some recent work. Potential theory and Green functions have also been considered~\cite{bass1991,kurkova_martin_1998}. Another quantity of interest is the probability of hitting the origin for a transient process, which in the insurance context can be interpreted as ruin in a model of two collaborating companies~\cite{albrecher_17,ivanovs_boxma}; see also \cite{hobson_rogers,varadhan_williams} for some fundamentals concerning the Brownian model.

In this paper we consider a bivariate L\'evy process with a negative mean in a non-standard regime, where the reflection vectors point away from the origin, forcing the reflected process to escape to infinity along one of the axes. We say that the first component totally dominates the second if the process escapes to infinity along the $x$-axis, that is, the first component grows to infinity while the second becomes relatively negligible; see Figure~\ref{fig:domination} for an illustration. Under rather general conditions, it is shown in Theorem~\ref{thm:domination} that one of the components dominates the other almost surely and that each component can be dominant with positive probability for any fixed initial position. Additionally, we establish an invariance principle in Theorem~\ref{thm:dom_prob_conv} justifying, for example, the use of the Brownian approximation in applications.

Some of the possible interpretations of our model include the following:
\begin{itemize}
\item
Two funds diminishing on average, with an agreement that deficit in one fund is instantaneously covered together with a proportional capital inflow in the other. This inflow may also result indirectly from the loss of rating or trust.

\item
Two coupled servers with a special feature that one server upon becoming idle hinders the other (or provides some extra work for the other).
\end{itemize}
We mainly think about the first interpretation and sometimes use the respective terminology, such as capital and injections.

\begin{figure}[ht!]
\begin{center}
\includegraphics[width=0.6\textwidth]{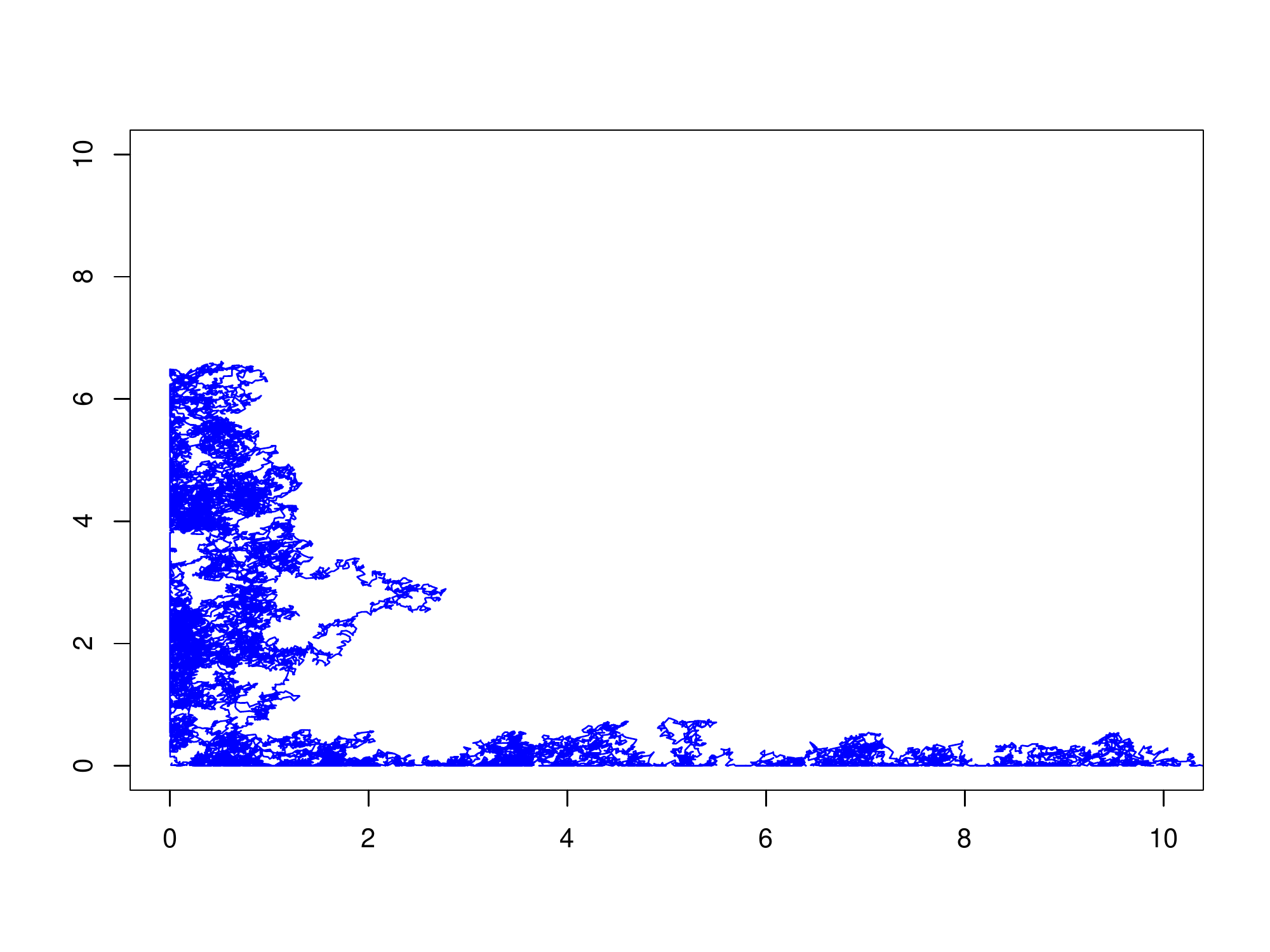}
\caption{Reflected Brownian motion started at the origin: domination of the first component.}
\label{fig:domination}
\end{center}
\end{figure}

It must be noted that the conditions imposed on the reflection angles lead to a non-unique solution of the Skorokhod problem in general, which makes the definition of the model problematic. We resolve this by restricting our attention to certain subclasses of bivariate L\'evy processes. Firstly, the model in the Brownian case is defined by~\cite{taylor_williams}, where the authors also showed its uniqueness in law and derived some important properties. Secondly, a simple iterative construction can be applied if one of the components of the free process does not become negative immediately. In particular, this allows for a compound Poisson process, where each component has a positive linear drift and only negative jumps (cf.\ the classical Cram\'er--Lundberg model in risk insurance). We stress that non-uniqueness and the particular implementation of reflection at the origin has no or little effect on our results. Furthermore, we formulate the domination and approximation results in such a way that other models can be added easily upon verification of some basic properties.

Additionally, we identify the Laplace transform of the probability that the first entity wins by totally dominating the second in two important special cases:
\begin{itemize}
\item[(i)]
the aforementioned compound Poisson model with independent components and negative exponential jumps;

\item[(ii)]
the correlated Brownian model.
\end{itemize}
Firstly, we derive a so-called kernel equation in case (i) additionally allowing for common jumps (shocks), and then obtain a kernel equation in case (ii) via approximation, relying on the theory developed below. While in case (ii) the kernel has already been studied in~\cite{franceschi_2019} for different equations/problems, in case (i) we have a completely new analytic problem. Even though our kernel equations resemble the one in~\cite{ivanovs_boxma}, the Wiener--Hopf methods from there seem not to be applicable in the current setting.

The kernel equations are solved by reducing them to the Carleman boundary value problem (BVP) following the general scheme presented in the classical monograph~\cite{fayolle_random_2017}. This method initially proposed in the seventies \cite{fayolle_two_1979,malysev_analytic_1972} has been used to study random walks in the quadrant, their invariant measures and Green functions \cite{kourkova_random_2011,kurkova_martin_1998}, and some related queueing models \cite{baccelli_analysis_1987}. This approach has also been fruitful in the continuous setting for computing the stationary distribution of a reflected Brownian motion in the quadrant \cite{franceschi_2019}. Our solutions are given in terms of a single contour integral along a half-circle in case (i), see Theorem~\ref{thm:BVPsolution}, and along a half-hyperbola in case (ii), see Theorem~\ref{thm:explicitF1Brown}. Furthermore, we obtain the probability of domination when starting at the origin and also derive some asymptotic results.

The paper is organized as follows. The model is defined in Section~\ref{sec:def}, and a basic result concerning the total domination probabilities is proven in Section~\ref{sec:dom}. The approximation result and its proof, relying on the uniform law of large numbers for L\'evy processes, are given in Section~\ref{sec:approx}. The kernel equations for models (i) and (ii) are derived in Section~\ref{sec:kernel} from the one for the Poissonian model with common shocks. With regard to the latter equation, we only summarize the basic steps, whereas the corresponding lengthy and tedious calculations are presented in Appendix~\ref{sec:derivation}. We solve the kernel equation for the Poissonian model (i) in Section~\ref{sec:poi} and for the Brownian model (ii) in Section~\ref{sec:BM}. Finally, numerical illustrations are provided in Section~\ref{sec:num}.

\section{Definition of the model}
\label{sec:def}
Consider a probability space with filtration $\mcF_t$ and let $X(t)=(X_1(t),X_2(t))$, $t\ges 0$, be an adapted bivariate L\'evy process, that is, a process with stationary and independent increments which is continuous in probability; without loss of generality, we assume that it has c\`adl\`ag paths without fixed jumps (e.g., see~\cite[Theorem~15.1]{kallenberg}). Our main examples will be a correlated Brownian motion and a drifted compound Poisson process, whose two components may exhibit both individual and common jumps.

\subsection{Skorokhod problem}
A bivariate process $Y\ges 0$ is a solution to the Skorokhod problem~\cite{williams_survey}, also known as the dynamic complementarity problem, if  the following holds a.s.:
\begin{equation}
\label{eq:definition}
\begin{split}
Y_1(t)=u+X_1(t)+L_1(t)+r_2L_2(t),\\
Y_2(t)=v+X_2(t)+r_1L_1(t)+L_2(t),
\end{split}
\end{equation}
where $(u,v)$ is the starting position with $u,v\ges 0$, and $L_i$ are the regulators (cumulative capital injections) satisfying
\begin{enumerate}[label=(\roman*)]
\item
\label{L1}
$L_i(t)$ are non-decreasing with $L_i(0)=0$,

\item
\label{L2}
$L_i(t)$ increases only when $Y_i(t)=0$, i.e., $\int_0^\infty Y_i(s)\D L_i(s)=0$.
\end{enumerate}
It is assumed that all the processes are adapted to the given filtration. The second condition concerns minimality of injections, meaning that no injections are received unless strictly necessary; in particular, we have
\begin{equation*}
\begin{gathered}
L_1(t)=\sup_{0\les s\les t} \left[-u-X_1(s)-r_2L_2(s)\right]\vee 0,\\
L_2(t)=\sup_{0\les s\les t} \left[-v-X_2(s)-r_1L_1(s)\right]\vee 0.
\end{gathered}
\end{equation*}

Differently to the classical setting, we assume that
\begin{equation}
\label{eq:r}
r_1,r_2>0\qquad \text{ and }\qquad r_1r_2>1.
\end{equation}
The corresponding reflection matrix $\left(\begin{smallmatrix} 1 & r_2\\ r_1 & 1\end{smallmatrix}\right)$ belongs to the so-called completely-$\mcS$ class and thus our Skorokhod problem has a solution in the sample-path sense~\cite{kozyakin1993absolute}. Uniqueness, however, is not guaranteed, leading to certain measurability issues for general processes, see~\cite{bernard_el_kharroubi,williams_survey}. Nevertheless, in the Brownian case there is a unique weak solution~\cite{taylor_williams}. Moreover,~\cite{williams_inv} establishes an invariance property allowing to retrieve the Brownian model as a weak limit of approximations on compact time intervals.

\subsection{Iterative definition and linear complementarity problem}
\label{sec:iterative}
To define the reflected process for a more general $X$, we need to recall an important dichotomy for one-dimensional L\'evy processes: the probability of immediate entrance into the negative half-line $(-\infty,0)$ is either 0 or~1. In the first case the entrance time is strictly positive and the main example is a process of bounded variation on compacts with a positive linear drift~\cite[Proposition~VI.11]{bertoin}.

Coming back to the bivariate process $X$, we assume that at least one of its components enters $(-\infty,0)$ at a strictly positive time. Without loss of generality, we assume that $X_2$ is such, and let $T_k$, $k\ges 1$, be the random times when $X_2$ (or, equivalently, $v+X_2$) updates its infimum; for convenience, we also set $T_0=0$. It is clear that if
\[
v+X_2(T_{k-1})+r_1L_1(T_{k-1})+L_2(T_{k-1})\ges 0,
\]
then, since $r_1>0$, we also have
\[
v+X_2(t)+r_1L_1(t)+L_2(T_{k-1})\ges 0,\quad T_{k-1}\les t<T_k;
\]
besides, $T_k\to\infty$ a.s., $k\to\infty$. Therefore, in order to obtain the reflected process $Y$ on $[0,+\infty)$, we just need to define it on the intervals $[T_{k-1},T_k)$, $k\ges 1$, keeping $L_2$ constant on them.

To this end, we set $L_2(t)=0$ for $t<T_1$, and then define $Y$ on $[0,T_1)$ by reflecting $u+X_1+r_2L_2$ in the one-dimensional sense up to $T_1$, i.e.~taking
\[
L_1(t)=-\inf_{0\les s\les t} [0\wedge (u+X_1(s))],\quad t<T_1.
\]
Then, loosely speaking, at the moment $T_1$ we solve the corresponding linear complementarity problem, reset $Y$ accordingly and proceed from there, repeating the procedure. More precisely, at each epoch $T_k$, which is a stopping time, we let
\[
x_i=Y_i(T_k-)+\Delta X_i(T_k),\quad i=1,2,
\]
where $\Delta X_i(T_k)=X_i(T_k)-X_i(T_k-)$, and solve the linear complementarity problem for this $x=(x_1,x_2)\in\mbR^2$:
\begin{equation}
\label{eq:linear}
y_1=x_1+\ell_1+r_2\ell_2,\qquad y_2=x_2+\ell_2+r_1\ell_1,
\end{equation}
where $y_i,\ell_i\ges 0$ and $\ell_iy_i=0$. Then we set $Y_i(T_k)=y_i$, $L_i(T_k)=L_i(T_k-)+\ell_i$, and proceed as if $Y(T_k)=(y_1,y_2)$ were the starting position instead of $(u,v)$ and $X(T_k+\cdot)-X(T_k)$ were the free process instead of $X$, whereas we let $L$ accumulate the needed future injections.

Thus defined processes $Y$, $L_1$ and $L_2$ are clearly adapted to the given filtration and satisfy~\eqref{eq:definition} together with \ref{L1} and \ref{L2}. Besides, if both $X_1$ and $X_2$ enter $(-\infty,0)$ at a strictly positive time, then $L_1$ and $L_2$ are piecewise constant and do not depend on the initial choice of the component of $X$ for which $T_k$'s are constructed.

However, it turns out that the static problem~\eqref{eq:linear} can have multiple solutions for certain~$(x_1,x_2)<0$. In principle, any of these can be used, and one may even pick a solution in an $\mcF_{T_k}$-measurable random way. However, we choose one specific solution, which we are now going to describe.

In the static problem~\eqref{eq:linear}, $x_i\ges 0$ necessarily implies that $\ell_i=0$. In particular, $x_1,x_2\ges 0$ yields $y_i=x_i$ (no adjustment). Furthermore, if $x_1<r_2x_2\wedge 0$, then $y_1=0$ and $y_2=x_2-r_1x_1$, whereas if $x_2<r_1x_1\wedge 0$, then $y_1=x_1-r_2x_2$ and $y_2=0$. The final case concerns the wedge:
\[
x_1,x_2<0,\quad x_1\ges r_2x_2,\quad x_2\ges r_1x_1,
\]
see also Figure~\ref{fig:cases}. Here we have three solutions (two on the boundary):
\begin{itemize}
\item[(i)]
$y_1=y_2=0$,

\item[(ii)]
$y_1=x_1-r_2x_2,\quad y_2=0$,

\item[(iii)]
$y_1=0,\quad y_2=x_2-r_1x_1$.
\end{itemize}
In the following we pick (i) for concreteness, which resets both components to $0$ when ambiguity arises. It is noted that this particular choice has no or little effect on our results, which we also stress in the following.

\begin{figure}[ht!]
\begin{center}
\begin{tikzpicture}
\fill[color=blue!50,opacity=0.25] (0,0)--(-3,-1)--(-3,-3)--(2,-3)--(2,0);
\fill[color=red!50,opacity=0.25] (0,0)--(-2,-3)--(-3,-3)--(-3,2)--(0,2);
\draw (-3,0)--(0,0);
\draw (0,-3)--(0,0);
\draw[thick] (2,0)--(0,0)--(0,2);
\node at (1,1) {$\mbR_+^2$};
\node[above] at (-2.5,0) {$x_1$};
\node[right] at (0,-2.5) {$x_2$};
\draw[blue!50] (0,0)--(-3,-1);
\draw[red!50] (0,0)--(-2,-3);
\draw[->,red] (-1.5,-0.25) to [out=60,in=180] (0,0.5);
\draw[->,blue] (-0.5,-1.5) to [out=30,in=-90] (0.5,0);
\node[above,red] at (-1,0.5) {$x_2-r_1x_1$};
\node[below,blue] at (1.3,0) {$x_1-r_2x_2$};
\end{tikzpicture}
\caption{Solutions to linear complementarity problem~\eqref{eq:linear}. The blue half-line corresponds to $x_1=r_2x_2<0$ and the red to $x_2=r_1x_1<0$. The red region results in $y_1=0$ and the blue region in $y_2=0$. The wedge corresponds to three solutions and for concreteness we choose $y_1=y_2=0$ there.}
\label{fig:cases}
\end{center}
\end{figure}
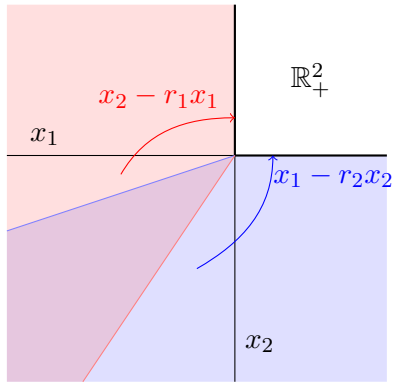

Finally, it is worth mentioning that in a similar way one can construct the reflected process for the sum of a Brownian motion and an arbitrary independent compound Poisson process, where between the jumps the model evolves as a reflected Brownian motion and at jump epochs we again solve~\eqref{eq:linear}.

\subsection{Basic properties}
Here we observe some basic properties of the reflected process. Firstly, note that the regulator does not increase when the free process is non-negative:
\begin{equation}
\label{eq:lzero}
\forall\, t\in [0,T]\colon\quad u+X_1(t)\ges 0\qquad \Longrightarrow\qquad L_1(T)=0,
\end{equation}
since from~\eqref{eq:definition} we then have $Y_1(t)\ges L_1(t)$ and thus $\int_0^T L_1(t)\D L_1(t)=0$. In such a case $L_2(t)=(-\inf_{0\les s\les t} [v+X_2(s)])^+$ and the expressions for $Y_1$ and $Y_2$ are straightforward. Unlike the classical case, however, non-uniqueness presents some problems: if $L_1(T)=0$ yields a non-negative solution (and even $Y_1$ may be strictly positive on $[0,T]$), then we cannot conclude that this is the right solution.

Importantly,
\begin{equation}
\label{eq:strong}
\text{$Y$ is strong Markov},
\end{equation}
so that for any finite stopping time $\tau$, conditional on $Y(\tau)=(u',v')$, the process $Y'(t)=Y(\tau+t)$ is independent of $\mcF_\tau$ and has the original law when started at $(u',v')$. In the Brownian case this is a consequence of the strong Feller property shown in~\cite{taylor_williams}, and in the case of the iterative construction of~\S\ref{sec:iterative} this property is obviously inherited from the process~$X$. Note, however, that the choice in~\eqref{eq:linear} must not depend on the future evolution of the process.

Finally we comment on rescaling of the model. For any $a_1,a_2>0$  by setting
\begin{equation}
\label{eq:rescalign}
X'_i(t)=a_iX_i(t),\quad u'=a_1u,\quad v'=a_2 v,\quad r_1'=\dfrac{a_2}{a_1}r_1,\quad r_2'=\dfrac{a_1}{a_2}r_2,
\end{equation}
we find that $Y'_i(t)=a_iY_i(t)$ with $L'_i(t)=a_iL_i(t)$ being a solution of~\eqref{eq:definition}. Furthermore, we resolve non-uniqueness in~\S\ref{sec:iterative} in a consistent way implying $Y_i'(t)=a_i Y_i(t)$. Thus, the probability of total domination defined in~\S\ref{sec:dom} is invariant under any such scaling given that the initial position is scaled appropriately.

\section{Domination}
\label{sec:dom}
\subsection{The result}
We assume throughout this paper that $X$ is a bivariate L\'evy process (with c\`adl\`ag paths) such that
\begin{gather}
\label{A1}
\tag{A1}
\e X(1)=\mu=(\mu_1,\mu_2)<0,\\
\label{A2}
\tag{A2}
r_1\abs*{\mu_1}>\abs*{\mu_2},\qquad r_2\abs*{\mu_2}>\abs*{\mu_1},
\end{gather}
where the latter implies~\eqref{eq:r}. Furthermore, we assume that the reflected process $Y$ is well-defined in the sense that it satisfies~\eqref{eq:definition} and~\eqref{eq:strong}. It is noted that in the Brownian case the above conditions imply that $Y$ is transient~\cite{hobson_rogers}, but more is true as we show in the following.

\begin{figure}[ht!]
\begin{center}
\begin{tikzpicture}[scale=1.3]
\draw[->] (0,0)--(2.5,0);
\draw[->] (0,0)--(0,2);
\draw[dashed] (1,0)--(1.5,0.25);
\draw[dashed] (0,1)--(0.5,1.25);
\draw[->,thick] (1.7,1.5) --(1.1,1.2);
\draw[->,thick,magenta] (1,0) -- (1.7,0.15);
\draw[->,thick,magenta] (0,1) --(0.25,1.5);
\node[above] at (1.7,1.5) {$(\mu_1,\mu_2)$};
\node[above,magenta] at (1.7,0.25) {$(r_2,1)$};
\node[above,magenta] at (0.5,1.5) {$(1,r_1)$};
\end{tikzpicture}
\caption{Reflection vectors and the mean.}
\label{fig:reflection_angles}
\end{center}
\end{figure}
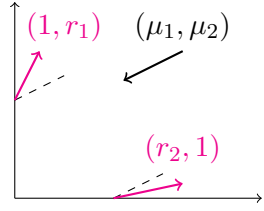

An additional technical assumption is needed to exclude certain degenerate cases:
\begin{equation}
\label{A3}
\tag{A3}
\p\{\exists\, t>0\colon\; X_i(t)>0,\; X_j(t)=\ulX_j(t)\}>0,\quad (i,j)=(1,2),\; (2,1),
\end{equation}
where $\ulX_j(t)\coloneqq\inf_{0\les s\les t} X_j(s)$. This condition is not minimal possible, but we avoid further technicalities since it is broadly satisfied. Importantly, for the Brownian model it is sufficient to assume that its correlation $\rho$ is not $1$. For the compound Poisson model with positive linear drift $c$ it is sufficient to assume that both components may exhibit individual negative jumps. As an example not satisfying~\eqref{A3} consider jumps distributed as $(\Delta_1,\Delta_2)$, where $\Delta_i<0$ and $\p(\Delta_1/\Delta_2\ges c_1/c_2)$ is either~$1$ or~$0$. It should be mentioned that such models with ordered jumps have been used, for example, in~\cite{badila2014queues}, because they allow for simpler analysis in various settings.

Our focus is on the probabilities $p_i=p_i(u,v)$ of total domination starting from $(u,v)$ which is defined by
\begin{gather*}
p_1(u,v)=\p_{(u,v)}\left\lbrace Y_1(t)\to\infty,\; \dfrac{Y_2(t)}{Y_1(t)}\to 0\right\rbrace,\\
p_2(u,v)=\p_{(u,v)}\left\lbrace Y_2(t)\to\infty,\; \dfrac{Y_1(t)}{Y_2(t)}\to 0\right\rbrace.
\end{gather*}

The following result shows that total domination is certain, and each component can be the dominant one for any starting position.

\begin{thm}[Total domination probabilities]
\label{thm:domination}
Under conditions~\eqref{A1}, \eqref{A2} and~\eqref{A3} we have for any $(u,v)\in\mbR^2_+$:
\[
p_1(u,v)\in (0,1)\qquad\text{ and }\qquad p_1(u,v)+p_2(u,v)=1.
\]
Moreover, $\lim\limits_{u\to\infty} p_1(u,v)=1$ and $\lim\limits_{v\to\infty} p_1(u,v)=0$.
\end{thm}

The proof of this result is based on two lemmas and an observation that $Y$ visits the boundary infinitely often. Firstly, we employ a regeneration argument to show that $Y$ hits the remote parts of the quadrant boundary almost surely. Secondly, when starting in those remote parts the process $Y$ has the claimed behaviour with high probability, which follows from the strong law of large numbers and some basic properties underlying~\eqref{eq:definition}.

\subsection{Proofs}
By the law of large numbers, we have
\begin{equation}
\label{eq:LLN}
\dfrac{X_i(t)}{t}\cas\mu_i,\quad t\to\infty,\quad i=1,2;
\end{equation}
see, e.g.,~\cite[Theorem~36.5]{sato}. This implies that if conditions~\eqref{A1} and~\eqref{A2} are satisfied, then the reflected stochastic process $Y$ hits the boundary $\partial\mbR_+^2$ infinitely often:
\[
\sup\{t\ges 0\colon\; Y_1(t)\wedge Y_2(t)=0\}=\infty\quad\as
\]
Indeed, suppose that $\tau=\sup\{t\ges 0\colon\; Y_1(t)=0\}<\infty$. By definition, we have $Y_1(t)>0$, $t>\tau$, and so $L_1(t)=L_1(t\wedge\tau)$, $t\ges 0$. Therefore, using~\eqref{eq:LLN}, we obtain
\[
\lim_{t\to\infty} \dfrac{L_2(t)}{t}=\lim_{t\to\infty} \dfrac{1}{t}\sup_{0\les s\les t} (-v-X_2(s)-r_1L_1(s\wedge\tau))^+=-\mu_2>0\quad\as,
\]
which implies that $\sup\{t\ges 0\colon\; Y_2(t)=0\}=\infty$.

However, if condition~\eqref{A3} is also fulfilled, then a stronger assertion holds true; namely, the reflected process hits the remote parts of the boundary $\partial\mbR_+^2$ almost surely.

\begin{lem}
\label{lem:remote}
Assume conditions~\eqref{A1}, \eqref{A2} and~\eqref{A3}, and for any $h>0$ define two disjoint sets
\[
D^1_h=\{(x,0)\colon\; x\ges h\},\qquad D^2_h=\{(0,y)\colon\; y\ges h\}.
\]
Then for any fixed $(u,v)\in\mbR_+^2$ and all $h>0$ the stochastic process $Y$ satisfies:
\begin{align*}
&\p_{(u,v)}\left\lbrace\exists\, t\ges 0\colon\; Y(t)\in D_h^1\cup D_h^2\right\rbrace=1,\\
&\p_{(u,v)}\left\lbrace\exists\, t\ges 0\colon\; Y(t)\in D_h^i\right\rbrace>0,\quad i=1,2.
\end{align*}
\end{lem}

\begin{proof}
Note that the law of large numbers~\eqref{eq:LLN} and condition~\eqref{A3} imply that the paths of $X_1$ and $X_2$ take both positive and negative values, and so are not monotone functions with probability one. Besides, by condition~\eqref{A3} we have
\[
\p\left\lbrace\exists\, t>0\colon\; X_1(t)>0,\; X_2(t)=\ulX_2(t)\right\rbrace>0.
\]
Furthermore, since $\mu_2<0$, we can add $\ulX_2(t)<0$ into this probability to get
\begin{equation}
\label{eq:temp_q}
\p\left\lbrace\exists\, t>0\colon\; X_1(t)>0,\; X_2(t)=\ulX_2(t)<0\right\rbrace>0.
\end{equation}

Fixing any $\delta>0$, we note that since $X_1$ is not non-increasing it can become arbitrarily large before becoming~$\les -\delta$. Thus, using the strong Markov property and applying~\eqref{eq:temp_q} sufficiently many times, we obtain
\[
\p\left\lbrace\exists\, t>0\colon\; X_1(t)>1,\; \ulX_1(t)>-\delta,\; X_2(t)= \ulX_2(t)<0\right\rbrace>0,
\]
and hence for some $T>0$
\begin{equation}
\label{eq:q}
c=\p\left\lbrace\exists\, t\in (0,T]\colon\; X_1(t)>1,\; \ulX_1(t)>-\delta,\; X_2(t)= \ulX_2(t)<0\right\rbrace>0.
\end{equation}

As was shown at the beginning of this subsection, the stochastic process $Y$ visits the boundary of $\mbR_+^2$ infinitely often. Assume that for some $\delta>0$ the process $Y$ visits $D_\delta^1\cup D_\delta^2$ infinitely often. Let us show, using a regeneration argument, that the same is then true for $\delta'=\delta+1$. Consider an increasing sequence of stopping times $\tau_1,\tau_2,\ldots$ defined as the successive visits of the set $D_\delta^1$ with at least $T$ time units in between:
\begin{gather*}
\tau_1=\inf\{t\ges 0\mid Y(t)\in D_\delta^1\},\\
\tau_{i+1}=\inf\{t\ges\tau_i+T\mid Y(t)\in D_\delta^1\},\quad i\ges 1.
\end{gather*}
For each $i$ such that $\tau_i<\infty$, let $u_i=Y_1(\tau_i)$, and consider the probability that $Y$ hits $D^1_{u_i+1}$ in $[\tau_i,\tau_i+T]$, but before $Y_1$ becomes less than or equal to $u_i-\delta$, which will mean that it hits $D^1_{\delta+1}$. This probability is constant for all $i$ and is given by~\eqref{eq:q}. Hence, the probability of not visiting $D^1_{\delta+1}$ is upper bounded by $(1-c)^{N_1}$, where $N_1$ is the number of $\tau_i<\infty$. The same is true for the other direction. Since at least one of $N_1,N_2$ is infinite, this implies that visiting $D^1_{\delta+1}\cup D^2_{\delta+1}$ is certain.

Also, we note that if only the origin is visited infinitely often, then we may apply a similar regeneration argument at the origin to get a contradiction. Therefore, the above argument proves the first claim.

To prove the second statement, we note that the probability of hitting the boundary at a point other than the origin is positive. Firstly, $Y_1$ must be positive, since $X_1$ is not non-increasing. But for a positive $u$ we may again apply~\eqref{eq:q}, showing that hitting the ray $(x,0)$, $x>0$, is possible. This also shows that hitting $D^1_h$ for any $h>0$ and any starting position $(u,v)$ occurs with positive probability. A similar argument holds for the other component, which completes the proof.
\end{proof}

The following lemma shows that if the initial capital of one of the companies is sufficiently large, then this company will dominate with probability close to one. Its proof is based on the law of large numbers~\eqref{eq:LLN} for L\'evy processes.

\begin{lem}
\label{lem:domination_limit}
If conditions~\eqref{A1} and~\eqref{A2} are satisfied, then for any $\ve>0$ there exists $u_0=u_0(\ve)\ges 0$ such that
\[
p_1(u,v)\ges 1-\ve
\]
for all $v\ges 0$ and $u\ges (r_2v)\vee u_0$. Also, a similar assertion holds true for $p_2$.
\end{lem}

\begin{proof}
We note that~\eqref{eq:LLN} implies
\[
\sup_{t\ges T} \dfrac{X_i(t)}{t}\cip\mu_i,\quad T\to\infty,\quad i=1,2.
\]
Therefore, fixing arbitrarily small $\ve>0$ (more precisely, we will later need that $\ve<\ve_0$, where $\ve_0=(r_2\abs*{\mu_2}-\abs*{\mu_1})/(2r_2+2)>0$), we can choose $T=T(\ve)>0$ such that
\[
\p\left\lbrace\sup_{t\ges T} \max_{i=1,2} \abs*{\dfrac{X_i(t)}{t}-\mu_i}<\ve\right\rbrace= \p\left\lbrace\max_{i=1,2} \sup_{t\ges T} \abs*{\dfrac{X_i(t)}{t}-\mu_i}<\ve\right\rbrace\ges 1-\dfrac{\ve}{2}.
\]
Then we have
\begin{equation}
\label{eq:estimate}
(\mu_i-\ve)t<X_i(t)<(\mu_i+\ve)t,\quad t\ges T,\quad i=1,2,
\end{equation}
with probability not less than $1-\ve/2$. Also, let $u_0>0$ be so large that
\[
\p(-\ulX_1(T)<u_0)\ges 1-\ve/2.
\]
In the rest of the proof we focus on the intersection of these two events, which has probability not less than $1-\ve$.

Now, fix arbitrary $v\ges 0$ and $u\ges (r_2v)\vee u_0$, consider the random time
\[
\tau=\inf\{t\ges 0 \mid L_1(t)>0\}\ges T,
\]
and let us show that actually $\tau=\infty$. Indeed, we first note that if $\tau=T$, then $L_1(\tau)=0$, because
\[
Y_1(\tau)=Y_1(T)\ges u+X_1(T)\ges u_0-(-\ulX_1(T))>0.
\]
Moreover, if $\tau>T$, then, by the definition of $\tau$, for any $T\les t<\tau$ we have $L_1(t)=0$ and, using~\eqref{eq:estimate}, obtain
\begin{equation}
\begin{gathered}
L_2(t)\ges\sup_{0\les s\les t} (-v-X_2(s)-r_1L_1(s))^+=\sup_{0\les s\les t} (-v-X_2(s))^+\ges\\
\ges\sup_{T\les s\les t} (-v-X_2(s))^+\ges (\abs*{\mu_2}-\ve)t-v.
\end{gathered}
\end{equation}
Hence, for such $t$ we have
\[
Y_1(t)\ges u+X_1(t)+r_2L_2(t)\ges u-(\abs*{\mu_1}+\ve)t+r_2(\abs*{\mu_2}-\ve)t-r_2v\ges (u-r_2v)+ct>0,
\]
where $c=(r_2\abs*{\mu_2}-\abs*{\mu_1})/2>0$. Furthermore, the fact that $X_1(\tau)\ges -(\abs*{\mu_1}+\ve)\tau$ and $L_2$ is monotone implies that this bound also holds true for $t=\tau$.

Therefore, in both cases we have $Y_1(\tau)>0$. Owing to the right continuity of $X_1$ and monotonicity of $L_1$ and $L_2$, we have $Y_1(t)>0$ for any $t\in (\tau,\tau+\delta)$ with sufficiently small $\delta>0$. This means that $L_1(t)=0$ for $t\in (\tau,\tau+\delta)$, which contradicts the definition of $\tau$.

Thus, we conclude that for $u\ges (r_2v)\vee u_0$ the stochastic process $Y_1$ stays positive at all times. So, for all $t\ges 0$ we have
\begin{gather*}
Y_1(t)=u+X_1(t)+r_2L_2(t),\qquad Y_2(t)=v+X_2(t)+L_2(t),\\
L_2(t)=\sup_{0\les s\les t} (-v-X_2(s))^+.
\end{gather*}
It is easy to check that
\[
\lim_{t\to\infty} \dfrac{L_2(t)}{t}=\abs*{\mu_2},\qquad \lim_{t\to\infty} \dfrac{Y_1(t)}{t}= r_2\abs*{\mu_2}-\abs*{\mu_1}>0,\qquad \lim_{t\to\infty} \dfrac{Y_2(t)}{t}=0.
\]
Hence, the event of interest is ensured with probability not less than $1-\ve$.

The same argument is valid for the corresponding assertion with $p_2$.
\end{proof}

\begin{proof}[Proof of Theorem~\ref{thm:domination}]
Fix arbitrary $u,v\ges 0$. For any $\ve>0$ choose $u_0=u_0(\ve)>0$ and $v_0=v_0(\ve)>0$ as in Lemma~\ref{lem:domination_limit}, set $h=u_0\vee v_0$, and consider
\[
\tau_1=\inf\left\lbrace t\ges 0\colon\; Y(t)\in D^1_h\right\rbrace,\qquad \tau_2=\inf\left\lbrace t\ges 0\colon\; Y(t)\in D^2_h\right\rbrace,
\]
which are stopping times with respect to the given filtration.

By Lemma~\ref{lem:remote} the event $\{\tau_1<\infty\}$ has positive probability, and on this event the shifted process $X'(t)=X(\tau_1+t)-X(\tau_1)$ has the same law as the original L\'evy process and is independent of the corresponding position $Y(\tau_1)\in D^1_{u_0}$ (see~\cite[Proposition~I.6]{bertoin}). Therefore, noting that
\[
p_1(u,v)=\p_{(u,v)}\left\lbrace\tau_1<\infty,\; Y_1(t)\to\infty,\; \dfrac{Y_2(t)}{Y_1(t)}\to 0\right\rbrace,
\]
we obtain, by Lemma~\ref{lem:domination_limit},
\begin{gather*}
p_1(u,v)=\e_{(u,v)}\left[\indf\{\tau_1<\infty\}\cdot\p_{Y(\tau_1)}\left\lbrace Y_1(t)\to\infty,\; \dfrac{Y_2(t)}{Y_1(t)}\to 0\right\rbrace\right]\\
\ges (1-\ve)\cdot\p_{(u,v)}\{\tau_1<\infty\},
\end{gather*}
and so
\begin{equation}
\label{eq:sandwich}
(1-\ve)\cdot\p_{(u,v)}\{\tau_1<\infty\}\les p_1(u,v)\les\p_{(u,v)}\{\tau_1<\infty\}.
\end{equation}
Similar bounds hold true for $p_2(u,v)$ and $\tau_2$. Hence, according to Lemma~\ref{lem:remote}, both $p_1$ and $p_2$ are positive, which proves the first assertion, and also $p_1+p_2\ges 1-\ve$, which, due to the arbitrariness of $\ve$, implies the second assertion.
\end{proof}

\section{Approximation}
\label{sec:approx}
\subsection{Assumptions}
Throughout this section we consider a sequence of bivariate L\'evy processes $X^{(n)}$ converging weakly to $X$ with respect to the Skorokhod $J_1$-topology~\cite[\S3.3]{whitt}. This is equivalent to
\begin{equation}
\label{eq:convd}
\tag{C1}
X^{(n)}(1)\cid X(1),
\end{equation}
or to the convergence of the L\'evy exponents~\cite[Theorem~15.17]{kallenberg}. Furthermore, we assume that also the means converge:
\begin{equation}
\label{eq:convmean}
\tag{C2}
\mu^{(n)}=\e X^{(n)}(1)\to\e X(1)=\mu,
\end{equation}
which is equivalent, in view of~\eqref{eq:convd}, to the uniform integrability of $X^{(n)}(1)$.

It is assumed that the reflected processes $Y$ and $Y^{(n)}$ are well-defined, so that they satisfy~\eqref{eq:definition} and~\eqref{eq:strong}. Now we may expect that
\begin{equation}
\label{eq:convY}
\tag{C3}
Y^{(n)}\cid Y\qquad \text{whenever}\quad \mbR_+^2\ni (u^{(n)},v^{(n)})\to (u,v),
\end{equation}
which is indeed broadly satisfied for our models, including the case when $Y$ is a reflected Brownian motion as shown by~\cite{williams_inv}. Nevertheless, some exceptions exist as we now describe. The degenerate case is given by a drifted compound Poisson process with linear drifts $c_i>0$ and jumps distributed as $(J_1,J_2)$, where
\begin{equation}
\label{eq:degenerate}
c_i=r_jc_j\text{ and }J_i-r_jJ_j\text{ has a point mass}
\end{equation}
for some $i\in\{1,2\}$ and $j\neq i$.

\begin{lem}[Convergence of reflected processes]
\label{lem:degenerate}
The convergence in~\eqref{eq:convd} implies~\eqref{eq:convY} in the following cases:
\begin{itemize}
\item
$Y$ is a reflected Brownian motion and~\eqref{eq:convmean} holds,

\item
$Y,Y^{(n)}$ are defined in~\S\ref{sec:iterative}, apart from the case where $X$ is a drifted compound Poisson process satisfying~\eqref{eq:degenerate}.
\end{itemize}
\end{lem}

\begin{proof}
The first statement is a consequence of~\cite[Theorem~4.1 and Proposition~4.2(III)]{williams_inv}, where uniform integrability and martingale property readily follow from~\eqref{eq:convmean}.

Next, we consider the iterative construction of the reflected process, and recall that the one-dimensional reflection is a continuous map~\cite[\S 13.5]{whitt}. It is important that we resolve non-uniqueness of~\eqref{eq:linear} in the same way for all processes; recall that we have chosen to restart the processes from the origin if ambiguity arises. Our reflection map is then continuous at sample paths requiring finitely many iterations and not hitting the boundary of the wedge right before the application of linear complementarity, see Figure~\ref{fig:cases}. It is thus sufficient to show that the boundary of the wedge is not hit at the time $T_1$ in the construction of the limit process~$Y$ with probability~$1$.

Suppose that this occurs with positive probability. Since the jumps of $X$ below some negative threshold are independent, we see that $Y(T_1-)$ must have a mass on some line parallel to one of the wedge boundaries. Furthermore, we may replace $T_1$ by an independent exponential time. Assume for a moment that $X$ is not compound Poisson, in which case the distribution of $X_t$ for any $t>0$ is continuous~\cite[Theorem~27.4]{sato}. Ignoring the reflection we easily derive a contradiction by taking $t$ small and projecting $X$ onto the perpendicular direction. This argument can be extended to the case when $X_1$ does not spend time at the boundary (the Lebesgue measure is~$0$). In the only other case we may look at $X_2-r_1X_1$ to get the contradiction. Finally, assume that $X$ is a compound Poisson. The only possibility here is that included into~\eqref{eq:degenerate}.
\end{proof}

\subsection{The result and its proof}
Let us now state the approximation result for the domination probabilities. In fact, we show continuous convergence in the sense that perturbations in the initial positions are also allowed. Importantly, \eqref{eq:convY} is equivalent to convergence of the reflected process on compact intervals of time, and thus convergence of the limiting quantities is not obvious.

\begin{thm}[Invariance principle]
\label{thm:dom_prob_conv}
Assume that $X$ satisfies conditions of Theorem~\ref{thm:domination}, and let $X^{(n)}$ be a sequence of bivariate L\'evy processes approximating $X$ so that~\eqref{eq:convd}, \eqref{eq:convmean} and~\eqref{eq:convY} hold. Then
\[
\lim_{n\to\infty} p_i^{(n)}(u^{(n)},v^{(n)})=p_i(u,v),\qquad i=1,2,
\]
whenever $\mbR_+^2\ni(u^{(n)},v^{(n)})\to (u,v)$. In particular, $p_i$ are continuous for such~$X$.
\end{thm}

The main ingredient of the proof is the following uniform law of large numbers for L\'evy processes.

\begin{lem}
\label{lem:uLLN}
Let $X,X^{(n)}$ be bivariate L\'evy processes satisfying~\eqref{eq:convd} and \eqref{eq:convmean}. Then
\begin{equation}
\label{eq:ULLN}
\lim_{T\to\infty} \limsup_{n\to\infty} \p\left\lbrace\sup_{t\ges T} \max_{i=1,2} \abs*{\dfrac{X_i^{(n)}(t)}{t}-\mu_i}>\ve\right\rbrace=0
\end{equation}
for any $\ve>0$.
\end{lem}

\begin{proof}
Without loss of generality, we consider the one-dimensional case and assume that  $\mu=0$. Let us show that the stochastic process $\{M_{-t}=X(t)/t,\; t>0\}$ is a martingale with respect to the filtration $\mcG_{-t}=\sigma\left\lbrace X(t+s),\; s\ges 0\right\rbrace$, i.e., that for any $t>0$ and $s\ges 0$
\begin{equation}
\label{eq:martingale}
\e\left[\left.\dfrac{X(t)}{t}\right| X(t+s)\right]=\dfrac{X(t+s)}{t+s}.
\end{equation}
By the right continuity of the sample paths, it is sufficient to take $t=m(t+s)/n$ for some integers $m\les n$. However, it is a standard fact that for i.i.d. $Z_i$ with finite first moment we have the identity
\[
\e[Z_1+\cdots+Z_m \mid Z_1+\cdots+Z_n]=\dfrac{m}{n}(Z_1+\cdots+Z_n),
\]
and taking $Z_i=X(i(t+s)/n)-X((i-1)(t+s)/n)$ we get~\eqref{eq:martingale}.

Now, by Doob's martingale inequality~\cite[Proposition~7.15]{kallenberg}, for any $T'>T$ we have
\begin{equation}
\label{eq:martingale_inequality}
\p\left\lbrace\sup_{t\in [T,T']} \dfrac{\abs*{X(t)}}{t} \ges\ve\right\rbrace\les \dfrac{1}{\ve}\cdot\dfrac{\e\abs*{X(T)}}{T},
\end{equation}
which, by passing to the limit, readily extends to the infinite time interval $[T,\infty)$.

Thus, to prove~\eqref{eq:ULLN}, it is sufficient to show that
\[
\lim_{T\to\infty} \limsup_{n\to\infty} \dfrac{\e\abs*{X^{(n)}(T)}}{T}=0.
\]
However, for a fixed $T$ we have $X^{(n)}(T)\cid X(T)$ as $n\to\infty$, which implies the convergence of the mean absolute values, because the families $X^{(n)}(1)$, $n\ges 1$, and thus also $X^{(n)}(T)$, $n\ges 1$, are uniformly integrable. Finally, from~\eqref{eq:martingale_inequality} with the infinite time interval $[T,\infty)$, it is easy to deduce that the family $\abs*{X(t)}/t$, $t\ges T$, is uniformly integrable, and so $\e\abs*{X(T)}/T\to 0$ as $T\to\infty$ (see also~\cite[Theorem~36.5]{sato}).
\end{proof}

\begin{proof}[Proof of Theorem~\ref{thm:dom_prob_conv}]
Fix $\ve>0$ and note that the bounds in~\eqref{eq:sandwich} hold for all large~$n$, since then the conditions~\eqref{A1} and~\eqref{A2} are satisfied. Note, however, that $u_0$ there depends on $n$. Nevertheless, we can choose $u_0^{(n)}=u_0$ independently of~$n$, see the proof of Lemma~\ref{lem:domination_limit}. This is so, because we may use the same $T$ according to Lemma~\ref{lem:uLLN}, but then $\underline X_1^{(n)}(T)\cid \underline X_1(T)$. Furthermore, the bounds in Lemma~\ref{lem:domination_limit} are also true if the set $D'_{u_0}=\{(u,v)\in\mbR_+^2\colon\; u\ges (r_2v)\vee u_0\}$ is replaced by $D_{u_0+\delta}^1$ for any $\delta>0$ as defined in Lemma~\ref{lem:remote}.

Let $p_1(T)$ be the probability that $Y$ hits $D_{u_0+1}^1$ on $[0,T]$ starting from $(u,v)$, and let $p_1^{(n)}(T)$ be the probability that $Y^{(n)}$ hits $D'_{u_0}$ on $[0,T]$ starting from $(u^{(n)},v^{(n)})$. We choose $T\ges 0$ so large that
\[
0\les\p_{(u,v)}\{\text{$Y$ hits $D_{u_0+1}^1$}\}-p_1(T)<\ve.
\]
Then, by~\eqref{eq:sandwich},
\[
p_1(u,v)\les\p_{(u,v)}\{\text{$Y$ hits $D_{u_0+1}^1$}\}<p_1(T)+\ve
\]
and
\[
p_1(u,v)\ges (1-\ve)\cdot\p_{(u,v)}\{\text{$Y$ hits $D_{u_0+1}^1$}\}\ges (1-\ve)\cdot p_1(T).
\]
Similarly,
\[
(1-\ve)\cdot p_1^{(n)}(T)\les p_1^{(n)}(u^{(n)},v^{(n)})<p_1^{(n)}(T)+\ve.
\]
By assumption~\eqref{eq:convY}, we have $Y^{(n)}\cid Y$ in $D([0,T])\times D([0,T])$, and so
\[
p^{(n)}_1(T)>p_1(T)-\ve
\]
for all large enough $n$. Therefore, for all large enough $n$ we obtain
\[
p_1(u,v)-p_1^{(n)}(u^{(n)},v^{(n)})<(p_1(T)+\ve)-(p_1(T)-\ve)(1-\ve)<3\ve.
\]
Similarly,
\[
p_2(u,v)-p_2^{(n)}(u^{(n)},v^{(n)})<3\ve,
\]
which, owing to Theorem~\ref{thm:domination} and the inequality $p_1^{(n)}(u^{(n)},v^{(n)})+p_2^{(n)}(u^{(n)},v^{(n)})\les 1$, implies that
\begin{gather*}
3\ve>p_1(u,v)-p_1^{(n)}(u^{(n)},v^{(n)})=1-p_2(u,v)-p_1^{(n)}(u^{(n)},v^{(n)})\ges\\
\ges p_2^{(n)}(u^{(n)},v^{(n)})-p_2(u,v)>-3\ve.
\end{gather*}

Thus, we conclude that $p_1^{(n)}(u^{(n)},v^{(n)})\to p_1(u,v)$, $n\to\infty$.
\end{proof}

\subsection{Poissonian approximation of Brownian motion}
\label{sec:example}
Here we consider an approximation of the correlated Brownian motion via compound Poisson processes that allow both common and individual jumps with exponential distribution. This model may be useful for financial applications.

Let $N,N_1,N_2$ be independent Poisson processes with rates $\lambda,\lambda_1,\lambda_2>0$ respectively, and let $J_k,J^{(1)}_k,J^{(2)}_k$, $k\ges 1$, be independent standard exponential random variables that are also independent of $N,N_1,N_2$. Consider a drifted compound Poisson process $X=(X_1,X_2)$ given by
\begin{equation}
\label{eq:Poissonian_model_definition}
X_i(t)=c_it-\frac{1}{\oq_i}\sum_{k=1}^{N(t)} J_k-\frac{1}{q_i}\sum_{k=1}^{N_i(t)} J^{(i)}_k,\quad i=1,2,
\end{equation}
where $c_i,q_i,\oq_i>0$ are fixed parameters. Note that $\oq_i$ scale the common jumps (shocks), whereas $q_1,q_2$ are the rate parameters of the individual exponential jumps.

The corresponding Laplace exponent $\psi(s_1,s_2)=\log\e e^{s_1 X_1(1)+s_2X_2(1)}$ is given by 
\begin{equation}
\label{eq:psi_common}
\psi(s_1,s_2)=s_1c_1+s_2c_2-(\lambda+\lambda_1+\lambda_2)+ \dfrac{\lambda}{1+s_1/\oq_1+s_2/\oq_1}+\dfrac{\lambda_1}{1+s_1/q_1}+ \dfrac{\lambda_2}{1+s_2/q_2}
\end{equation}
for $s_1,s_2\ges 0$. Differentiating $\psi$ twice, we readily obtain:
\begin{gather*}
\e X_i(1)=c_i-\lambda/\oq_i-\lambda_i/q_i,\\
\var(X_i(1))=2\lambda/\oq_i^2+2\lambda_i/q_i^2,\\
\cov(X_1(1),X_2(1))=2\lambda/(\oq_1\oq_2).
\end{gather*}

\begin{lem}[Approximation of Brownian motion]
\label{lem:approximation}
For any $\sigma_i>0$, $\mu_i\in\mbR$ and $\rho\in [0,1]$ there exist parameters $c_i,q_i,\oq_i,\lambda_i,\lambda>0$ such that
\begin{align*}
\e X_i(1)=\mu_i,\quad\var(X_i(1))=\sigma_i^2,\quad \cov(X_1(1),X_2(1))=\rho\sigma_1\sigma_2.
\end{align*}
This is also true for a drifted compound Poisson process $X^{(n)}$ with parameters  
\begin{equation}
\label{approximation_parameters}
\begin{split}
\lambda^{(n)}=\lambda n,\; \lambda^{(n)}_i=\lambda_in,\;
\oq^{(n)}_i=\oq_i\sqrt{n},\; q_i^{(n)}=q_i\sqrt{n},\\
c^{(n)}_i=\mu_i+(\lambda/\oq_i+\lambda_i/q_i)\sqrt{n},
\end{split}
\end{equation}
and thus defined $X^{(n)}$ converge weakly, as $n\to\infty$, to the Brownian motion with means $\mu_i$, variances $\sigma_i^2$, and correlation~$\rho$.
\end{lem}

\begin{proof}
It is enough to take parameters such that
\[
\lambda/\oq_i^2=\dfrac{1}{2}\rho\sigma_i^2,\qquad \lambda_i/q_i^2=\dfrac{1}{2}(1-\rho)\sigma_i^2
\]
with $\lambda$, $\lambda_1$ and $\lambda_2$ large enough for $c_1=\mu_1+\lambda/\oq_1+\lambda_1/q_1$ and $c_2=\mu_2+\lambda/\oq_2+\lambda_2/q_2$ to be positive. Straightforward calculation shows that
\begin{gather*}
\psi^{(n)}(s_1,s_2)\to\dfrac{1}{2}\left(\sigma_1^2s_1^2+2\rho\sigma_1\sigma_2s_1s_2+ \sigma_2^2s_2^2\right)+\mu_1s_1+\mu_2s_2,
\end{gather*}
and so we have $X^{(n)}\cid W$ according to~\cite[Theorem~15.17]{kallenberg}, where $W$ is a Brownian motion with the given parameters.
\end{proof}

In conclusion, the above defined drifted compound Poisson processes $X^{(n)}$ with exponential jumps can be used to approximate a given Brownian motion $X$ with non-negative correlation $\rho\in [0,1)$ and means satisfying~\eqref{A1} and~\eqref{A2}, with~\eqref{A3} being automatic. The construction of $Y^{(n)}$ is straightforward, see~\S\ref{sec:iterative}, and the conditions of Theorem~\ref{thm:dom_prob_conv} are satisfied. Thus, the total domination probabilities for $X$ can be derived from those for~$X^{(n)}$, which we indeed use to derive the Brownian kernel equation in the next section.

\section{Kernel equations}
\label{sec:kernel}
In the following we study the total domination probability $p_1(u,v)$ for two basic models. In fact, our focus is on the Laplace transform of $p_1$ and its restrictions where one initial position is fixed at~$0$:
\begin{align}
\label{eq:Laplace_transforms}
\begin{split}
&F(s_1,s_2)=\iint_{\mbR_+^2} e^{-s_1u-s_2v}p_1(u,v)\D u\D v,\\
&F_1(s_1)=\int\limits_0^\infty e^{-s_1u}p_1(u,0)\D u,\quad F_2(s_2)=\int\limits_0^\infty e^{-s_2v}p_1(0,v)\D v,
\end{split}
\end{align}
where $s_1,s_2>0$. It is noted that
\[
\hat{F}(s_1,s_2)=s_1s_2F(s_1,s_2)
\]
can be seen as the total domination probability of the first component when starting at independent exponential positions with rates $s_1$ and~$s_2$. Moreover, $\hat{F}(s_1,s_2)\to s_1F_1(s_1)$, $s_2\to\infty$, noting that $p_1$ is continuous by Theorem~\ref{thm:dom_prob_conv}, apart from the case~\eqref{eq:degenerate}.

Finally, we observe that rescaling of the model in~\eqref{eq:rescalign} results in $\hat{F}'(s_1,s_2)=\hat{F}(a_1s_1,a_2s_2)$. This, for example, allows to assume that $\mu_1'=\mu_2'=-1$ by taking $a_i=1/\abs*{\mu_i}$, in which case~\eqref{A2} reads simply $r_i'>1$. Alternatively, in the Brownian model we may take $\sigma_i=1$ without any loss of generality.

\subsection{Compound Poisson model}
First, we consider the compound Poisson model from~\S\ref{sec:example} with independent drivers $X_i$ having positive linear drifts $c_i$, jump arrival rates $\lambda_i$ and the jumps being negative exponentials with rates $q_i$. The bivariate Laplace exponent of $(X_1,X_2)$ is thus given by
\begin{equation}
\label{eq:psi}
\psi(s_1,s_2)=c_1s_1+c_2s_2-\lambda_1-\lambda_2+\dfrac{\lambda_1}{1+s_1/q_1}+ \dfrac{\lambda_2}{1+s_2/q_2}.
\end{equation}
Note that the choice of solution in~\eqref{eq:linear} does not play a role in this case.

\begin{prop}[Poissonian kernel equation]
\label{prop:Poisson_kernel}
Let the Laplace exponent $\psi$ be given by~\eqref{eq:psi} with $c_i,\lambda_i,q_i>0$ being such that~\eqref{A1} and~\eqref{A2} are satisfied with $\mu_i=c_i-\lambda_i/q_i$ and some $r_i>0$. Then
\begin{align}
\label{kernel_equation_compound_Poisson_model}
\begin{split}
\psi(s_1,s_2)F(s_1,s_2)=&\ \psi_1(s_1,s_2)\left[F_1(s_1)-F_1\left(q_2/r_2\right)\right]+\\
+&\ \psi_2(s_1,s_2)\left[F_2(s_2)-F_2\left(q_1/r_1\right)\right]+F_0,
\end{split}
\end{align}
where
\begin{gather*}
\psi_1(s_1,s_2)=c_2-\dfrac{\lambda_2q_2}{(q_2+s_2)(q_2-r_2s_1)},\quad
\psi_2(s_1,s_2)=c_1-\dfrac{\lambda_1q_1}{(q_1+s_1)(q_1-r_1s_2)},\\
F_0=c_2F_1\left(q_2/r_2\right)+c_1F_2\left(q_1/r_1\right).
\end{gather*}
\end{prop}

It is important to note here that the kernel equation is explicit thanks to the assumption of exponential jumps. A more general (and cumbersome) kernel equation is discussed in \S\ref{sec:CPP_common}, where the common shocks are allowed. This particular equation is an important special case of Proposition~\ref{kernel_equation_Poissonian_approximation}.

Notice that the kernel equation of Proposition~\ref{prop:Poisson_kernel} (as well as the one of Proposition~\ref{prop:Brownian_kernel}) can have many solutions. Actually, it seems possible to obtain the same kernel equation for the Laplace transforms of $\p_{(u,v)}(A)$ with $A\in\bigcap_{t\ges 0}\sigma(Y(s),\; s\ges t)$, but a rigorous proof of this generalisation involves certain difficulties connected with the continuity and differentiability of $\p_{(u,v)}(A)$ that are hard to overcome. Thus, the kernel equation is a necessary condition for $p_1(u,v)$, but not a sufficient one. The uniqueness of the solution will be obtained in the following sections assuming the limit properties of Theorem~\ref{thm:domination}.

Next, we determine the constant $F_0$ which also yields a simple expression for $\hat{F}(q_2/r_2,q_1/r_1)$. For this purpose, we define the points
\begin{equation}
\label{eq:x0y0}
x_0\coloneqq\dfrac{\lambda_1}{c_1}-q_1>0
\quad \text{and}\quad
y_0\coloneqq\dfrac{\lambda_2}{c_2}-q_2>0
\end{equation}
which satisfy
\[
\psi(x_0,0)=\psi_2(x_0,0)=0,\quad \psi(0,y_0)=\psi_1(0,y_0)=0,\quad \text{and}\quad \psi(x_0,y_0)=0,
\]
see also Figure~\ref{fig:zoom} below.

\begin{lem}
\label{lem:F0}
In the setting of Proposition~\ref{prop:Poisson_kernel} we have
\begin{equation}
\label{eq:C0}
F_0=\dfrac{r_1(r_2\abs*{\mu_2}-\abs*{\mu_1})}{r_1r_2-1}\left(\dfrac{c_1}{q_1\abs*{\mu_1}}+ \dfrac{r_2c_2}{q_2\abs*{\mu_2}}\right)>0.
\end{equation}
\end{lem}

\begin{proof}
The limits in Theorem~\ref{thm:domination} imply that $\hat{F}(0+,y_0)=\hat{F}_1(0+)=1$ and $\hat{F}(x_0,0+)=\hat{F}_2(0+)=0$. Evaluating the kernel equation~\eqref{kernel_equation_compound_Poisson_model} at three points $(x_0,0+)$, $(0+,y_0)$ and $(x_0,y_0)$ we obtain the equalities:
\begin{align}
\label{eq:valueC}
0&=\psi_1(x_0,0)\left[F_1(x_0)-F_1\left(q_2/r_2\right)\right]+F_0,\\
\notag
\dfrac{c_1-\lambda_1/q_1}{y_0}&=-\dfrac{r_2c_2}{q_2}+\psi_2(0,y_0) \left[F_2(y_0)-F_2\left(q_1/r_1\right)\right]+F_0,\\
\notag
0&=\psi_1(x_0,y_0)\left[F_1(x_0)-F_1\left(q_2/r_2\right)\right]
+\psi_2(x_0,y_0)\left[F_2(y_0)-F_2\left(q_1/r_1\right)\right]+F_0.
\end{align}
We can now express $F_0$:
\[
F_0=\left(\dfrac{c_1-\lambda_1/q_1}{y_0}+\dfrac{r_2c_2}{q_2}\right) \dfrac{\psi_2(x_0,y_0)}{\psi_2(0,y_0)} \Biggm/ \left(\dfrac{\psi_1(x_0,y_0)}{\psi_1(x_0,0)}+ \dfrac{\psi_2(x_0,y_0)}{\psi_2(0,y_0)}-1\right),
\]
which upon simplification yields the stated expression.
\end{proof}

Importantly, the kernel equation~\eqref{kernel_equation_compound_Poisson_model} can be rewritten in a homogeneous form:
\begin{equation}
\label{eq:kerneleqsimple}
\psi(s_1,s_2)f(s_1,s_2)=\psi_1(s_1,s_2)f_1(s_1)+\psi_2(s_1,s_2)f_2(s_2),
\end{equation}
where the new functions are given by
\begin{gather}
\label{eq:F0tilde}
f(s_1,s_2)=F(s_1,s_2)-\dfrac{F_0/\widetilde{F}_0}{s_1 s_2},\quad \widetilde{F}_0=\dfrac{c_1r_1}{q_1}+\dfrac{c_2 r_2}{q_2},\\
\notag
f_1(s_1)=F_1(s_1)-F_1(q_2/r_2)-\dfrac{F_0}{\widetilde{F}_0} \left(\dfrac{1}{s_1}-\dfrac{r_2}{q_2}\right),\\
f_2(s_2)=F_2(s_2)-F_2(q_1/r_1)-\dfrac{F_0 }{\widetilde{F}_0} \left(\dfrac{1}{s_2}-\dfrac{r_1}{q_1}\right).
\end{gather}
This follows by realizing that
\begin{equation*}
\psi(s_1,s_2)\dfrac{1}{s_1s_2}=\psi_1(s_1,s_2)\dfrac{1}{s_1}+\psi_2(s_1,s_2)\dfrac{1}{s_2}+ \widetilde{F}_0,
\end{equation*}
multiplying it by $F_0/\widetilde{F}_0$, and subtracting from the kernel original equation.

\subsection{Correlated Brownian motion}
Secondly, we consider a correlated Brownian motion $X$ with means $\mu_i<0$, variances $\sigma_i^2>0$ and correlation $\rho\in [0,1)$, so that
\begin{equation}
\label{eq:psiBM}
\psi(s_1,s_2)=\dfrac{1}{2}(\sigma_1^2s_1^2+2\rho\sigma_1\sigma_2s_1s_2+\sigma_2^2s_2^2)+ \mu_1s_1+\mu_2s_2.
\end{equation}
We exclude $\rho=1$, because of condition~\eqref{A3}, and $\rho<0$ is likely to be similar but requires another approximating model and respective tedious analysis. Again, the ambiguity present in~\eqref{eq:linear} does not arise.

\begin{prop}[Brownian kernel equation]
\label{prop:Brownian_kernel}
Let the Laplace exponent $\psi$ be given by~\eqref{eq:psiBM} with $\mu_i<0$ satisfying~\eqref{A2} and $\rho\in [0,1)$. Then
\begin{equation}
\label{Brownian_kernel_equation}
\psi(s_1,s_2)F(s_1,s_2)=\psi_1(s_1,s_2)F_1(s_1)+\psi_2(s_1,s_2)F_2(s_2)+cp_1(0,0),
\end{equation}
where
\begin{gather}
\notag
\psi_1(s_1,s_2)=\mu_2+\dfrac{1}{2}\sigma_2^2(s_2-r_2s_1)+\rho\sigma_1\sigma_2s_1,\\
\notag
\psi_2(s_1,s_2)=\mu_1+\dfrac{1}{2}\sigma_1^2(s_1-r_1s_2)+\rho\sigma_1\sigma_2s_2,\\
\label{eq:defc}
c=\dfrac{1}{2}(r_1\sigma_1^2+r_2\sigma_2^2)-\rho\sigma_1\sigma_2.
\end{gather}
\end{prop}

The proof of this proposition is given in Subsection~\ref{subsec:derivationBrowniankernel}.

Interestingly, here and in Proposition~\ref{prop:Poisson_kernel} the quantities $\psi_i$ can be expressed as $\psi_1(s_1,s_2)=(\psi(s_1,s_2)-\psi(s_1,-r_2s_1))/(s_2+r_2s_1)$, which are the same as in~\cite{ivanovs_boxma} studying the probabilities of hitting the origin in a different regime.

Importantly, the above kernel equation implies a simple formula for the domination probability when starting at the origin, but only in the independent case. For later use define
\begin{equation}
\label{eq:x0y0Brownian}
x_0\coloneqq -\dfrac{2\mu_1}{\sigma_1^2}>0
\quad \text{and}\quad
y_0\coloneqq -\dfrac{2\mu_2}{\sigma_2^2}>0,
\end{equation}
which satisfy $\psi(x_0,0)=\psi_2(x_0,0)=0$ and $\psi(0,y_0)=\psi_1(0,y_0)=0$. Importantly, for $\rho=0$ we also have $\psi(x_0,y_0)=0$.

\begin{cor}
\label{cor:BMp00}
In the setting of Proposition~\ref{prop:Brownian_kernel} with $\rho=0$ there is the formula
\begin{equation}
\label{eq:BMp00}
p_1(0,0)=\dfrac{r_1(r_2\abs*{\mu_2}-\abs*{\mu_1})(\sigma_1^2\abs*{\mu_2}+r_2\sigma_2^2\abs*{\mu_1})} {\abs*{\mu_1}\abs*{\mu_2}(r_1r_2-1)(r_1\sigma_1^2+r_2\sigma_2^2)}.
\end{equation}
\end{cor}

\begin{proof}
We again use the limits $\hat{F}(0+,y_0)=\hat{F}_1(0+)=1$ and $\hat{F}(x_0,0+)=\hat{F}_2(0+)=0$. Evaluating the kernel equation~\eqref{Brownian_kernel_equation} at three points $(x_0,0+)$, $(0+,y_0)$ and $(x_0,y_0)$ we obtain the equalities:
\begin{align}
\notag
0&=\psi_1(x_0,0)F_1(x_0)+cp_1(0,0),\\
\label{eq:BM2}
\dfrac{\mu_1}{y_0}&=-\dfrac{r_2}{2}\sigma_2^2+\psi_2(0,y_0)F_2(y_0)+cp_1(0,0),\\
\notag
0&=\psi_1(x_0,y_0)F_1(x_0)+\psi_2(x_0,y_0)F_2(y_0)+cp_1(0,0).
\end{align}
It is left to express $p_1(0,0)$ and to simplify the final formula.
\end{proof}

Finally, we can rewrite the kernel equation~\eqref{Brownian_kernel_equation} in a homogeneous form:
\begin{equation}
\label{eq:kerneleqsimpleBrownian}
\psi(s_1,s_2)f(s_1,s_2)=\psi_1(s_1,s_2)f_1(s_1)+\psi_2(s_1,s_2)f_2(s_2),
\end{equation}
where the new functions are given by
\begin{equation}
\label{eq:f1BM}
\begin{split}
f(s_1,s_2)\coloneqq F(s_1,s_2)-\dfrac{p_1(0,0)}{s_1 s_2},\\
f_1(s_1)\coloneqq F_1(s_1)-\dfrac{p_1(0,0)}{s_1},\quad f_2(s_2)\coloneqq F_2(s_2)-\dfrac{p_1(0,0)}{s_2}.
\end{split}
\end{equation}

\subsection{Common jumps}
\label{sec:CPP_common}
Here we consider the compound Poisson model with common jumps/shocks described in~\S\ref{sec:example}. Importantly, \eqref{eq:degenerate} is only satisfied if both
\begin{equation}
\label{eq:toexclude}
c_i=r_jc_j\qquad \text{and}\qquad \oq_i=\oq_j/r_j
\end{equation}
for some $i\neq j$. Hence, apart from this case the probability $p_1(u,v)$ is continuous.

\begin{prop}
\label{kernel_equation_Poissonian_approximation}
Consider $X$ defined in~\eqref{eq:Poissonian_model_definition}, where $\lambda\ges 0$ and the means $\mu_i=c_i-\lambda/\oq_i-\lambda_i/q_i<0$ satisfy~\eqref{A2}, but~\eqref{eq:toexclude} is not true for both $i\neq j$.
\begin{itemize}
\item
If $r_1/\oq_1>1/\oq_2$ and $r_2/\oq_2>1/\oq_1$, then the following kernel equation is satisfied:
\begin{multline}
\label{Poissonian_kernel_equation}
\psi(s_1,s_2)F(s_1,s_2)=\psi_1(s_1,s_2)F_1(s_1)+ \psi_2(s_1,s_2)F_2(s_2)+\\
+\psi_3(s_1,s_2)F_1\left(\dfrac{r_1+s_1(r_1/\oq_1-1/\oq_2)}{(r_1r_2-1)/\oq_2}\right)+ \psi_4(s_1,s_2)F_2\left(\dfrac{r_2+s_2(r_2/\oq_2-1/\oq_1)}{(r_1r_2-1)/\oq_1}\right)+\\
+\psi_5(s_1,s_2)F_1\left(q_2/r_2\right)+\psi_6(s_1,s_2)F_2\left(q_1/r_1\right)+ \psi_0(s_1,s_2)p_1(0,0),
\end{multline}
where $\psi$ is given in~\eqref{eq:psi_common} and 
\begin{gather*}
\psi_0(s_1,s_2)=-\dfrac{\lambda[(r_1/\oq_1-1/\oq_2)(r_2/\oq_2-1/\oq_1) (1+s_1/\oq_1+s_2/\oq_2)+r_1/\oq_1^2+r_2/\oq_2^2-2/(\oq_1\oq_2)]}
{(1+s_1/\oq_1+s_2/\oq_2)(r_1+(r_1/\oq_1-1/\oq_2)s_1)(r_2+(r_2/\oq_2-1/\oq_1)s_2)},\\
\psi_1(s_1,s_2)=c_2-\dfrac{\lambda/\oq_2}{(1+s_1/\oq_1+s_2/\oq_2)(1-(r_2/\oq_2-1/\oq_1)s_1)}- \dfrac{\lambda_2/q_2}{(1+s_2/q_2)(1-r_2s_1/q_2)},\\
\psi_2(s_1,s_2)=c_1-\dfrac{\lambda /\oq_1}{(1+s_1/\oq_1+s_2/\oq_2) (1-(r_1/\oq_1-1/\oq_2)s_2)}-\dfrac{\lambda_1/q_1}{(1+s_1/q_1)(1-r_1s_2/q_1)},\\
\psi_3(s_1,s_2)=\dfrac{\lambda/\oq_2}{(1+s_1/\oq_1+s_2/\oq_2)(1-(r_2/\oq_2-1/\oq_1)s_1)},\\
\psi_4(s_1,s_2)=\dfrac{\lambda/\oq_1}{(1+s_1/\oq_1+s_2/\oq_2)(1-(r_1/\oq_1-1/\oq_2)s_2)},\\
\psi_5(s_1,s_2)=\dfrac{\lambda_2/q_2}{(1+s_2/q_2)(1-r_2s_1/q_2)},\\
\psi_6(s_1,s_2)=\dfrac{\lambda_1/q_1}{(1+s_1/q_1)(1-r_1s_2/q_1)}.
\end{gather*}

\item
If $r_1/\oq_1>1/\oq_2$ and $r_2/\oq_2\les 1/\oq_1$, then the following kernel equation is satisfied:
\begin{multline*}
\psi(s_1,s_2)F(s_1,s_2)=\psi_1(s_1,s_2)F_1(s_1)+ \psi_2(s_1,s_2)F_2(s_2)+\\
+\psi_3(s_1,s_2)F_1\left(\dfrac{r_1+s_1(r_1/\oq_1-1/\oq_2)}{(r_1r_2-1)/\oq_2}\right)+ \psi_4(s_1,s_2)F_2\left(\dfrac{1-(r_2/\oq_2-1/\oq_1)s_1}{(r_1r_2-1)/\oq_2}\right)+\\
+\psi_5(s_1,s_2)F_1\left(q_2/r_2\right)+\psi_6(s_1,s_2)F_2\left(q_1/r_1\right)+\\
+\psi_7(s_1,s_2)F_2(1/(r_1/\oq_1-1/\oq_2))+\psi_0(s_1,s_2)p_1(0,0),
\end{multline*}
where $\psi$ is given in~\eqref{eq:psi_common}, $\psi_1$, $\psi_2$, $\psi_3$, $\psi_5$ and $\psi_6$ are the same as above, and
\begin{gather*}
\psi_0(s_1,s_2)=-\dfrac{\lambda(r_1r_2-1)/\oq_2^2} {(1+s_1/\oq_1+s_2/\oq_2)(r_1+(r_1/\oq_1-1/\oq_2)s_1)(1-(r_2/\oq_2-1/\oq_1)s_1)},\\
\psi_4(s_1,s_2)=\dfrac{\lambda/\oq_2}{(1+s_1/\oq_1+s_2/\oq_2)(r_1+(r_1/\oq_1-1/\oq_2)s_1)},\\
\psi_7(s_1,s_2)=\dfrac{\lambda(r_1/\oq_1-1/\oq_2)} {(r_1+(r_1/\oq_1-1/\oq_2)s_1)(1-(r_1/\oq_1-1/\oq_2)s_2)}.
\end{gather*}

\item
If $r_1/\oq_1\les 1/\oq_2$ and $r_2/\oq_2>1/\oq_1$, then the kernel equation coincides with that for the previous case with the indices changed correspondingly.
\end{itemize}
\end{prop}

The derivation is tedious and thus is postponed to Appendix~\ref{sec:derivation}. It is based on the analysis of all the non-negligible scenarios on the infinitesimal time interval $[0,h]$ and the strong Markov property. Then we take transforms and the limit as $h\downarrow 0$, which are followed by lengthy algebraic manipulations. It is important here that the probability $p_1$ is continuous as mentioned above.

Note that the kernel equation~\eqref{kernel_equation_compound_Poisson_model} follows immediately from~\eqref{Poissonian_kernel_equation} by taking $\lambda=0$, where every case can be used, since $\oq_i$ are arbitrary.

\subsection{Derivation of the Brownian kernel by approximation}
\label{subsec:derivationBrowniankernel}
The proof of~\eqref{Brownian_kernel_equation} is based on the approximation in~\S\ref{sec:example}.

\begin{proof}[Proof of Proposition~\ref{prop:Brownian_kernel}]
Let us choose the approximating models as specified in Lemma~\ref{lem:approximation}, and consider the sequence of kernel equations in~\eqref{Poissonian_kernel_equation}. Importantly, we can always avoid the degenerate case in~\eqref{eq:toexclude} for each~$n$; in addition, considering here, for the sake of brevity, only the case when $r_1\sigma_1>\sigma_2$ and $r_2\sigma_2>\sigma_1$, we can also choose the approximating parameters such that $r_1/\oq_1^{(n)}>1/\oq_2^{(n)}$ and $r_2/\oq_2^{(n)}>1/\oq_1^{(n)}$.

Now we recall that $\psi^{(n)}(s_1,s_2)\to\psi(s_1,s_2)$, and by Theorem~\ref{thm:dom_prob_conv} and the dominated convergence theorem, we have $F^{(n)}(s_1,s_2)\to F(s_1,s_2)$ and $F^{(n)}_i(s_i)\to F_i(s_i)$ for $i=1,2$. Also, it is easy to check that
\begin{gather*}
\psi_0^{(n)}(s_1,s_2)\to -\rho\sigma_1\sigma_2+ \dfrac{\rho}{2}\left(\dfrac{\sigma_1^2}{r_2}+\dfrac{\sigma_2^2}{r_1}\right),\\
\psi_i^{(n)}(s_1,s_2)\to \mu_i+\dfrac{1}{2}\sigma_i^2(s_i-r_is_j)+ \rho\sigma_1\sigma_2s_j,\qquad  (i,j)=(1,2)\; \text{ or }\; (2,1).
\end{gather*}
Furthermore,
\begin{gather*}
\psi_3^{(n)}(s_1,s_2)F_1^{(n)}\left(\dfrac{r_1\oq_2^{(n)}+s_1(r_1\oq_2^{(n)}/\oq_1^{(n)}-1)} {r_1r_2-1}\right)\to\dfrac{\rho\sigma_2^2}{2}\cdot\dfrac{r_1r_2-1}{r_1}\cdot p_1(0,0),\\
\psi_4^{(n)}(s_1,s_2)F_2^{(n)}\left(\dfrac{r_2\oq_1^{(n)}+s_2(r_2\oq_1^{(n)}/\oq_2^{(n)}-1)} {r_1r_2-1}\right)\to\dfrac{\rho\sigma_1^2}{2}\cdot\dfrac{r_1r_2-1}{r_2}\cdot p_1(0,0),
\end{gather*}
and
\begin{gather*}
\psi_5^{(n)}(s_1,s_2)F_1\left(q_2^{(n)}/r_2\right)\to \dfrac{1}{2}(1-\rho)\sigma_2^2r_2p_1(0,0),\\
\psi_6^{(n)}(s_1,s_2)F_2\left(q_1^{(n)}/r_1\right)\to \dfrac{1}{2}(1-\rho)\sigma_1^2r_1p_1(0,0).
\end{gather*}

Combining the obtained values we arrive at the stated result. All other cases can be considered in a similar way and lead to the same kernel equation.
\end{proof}

\section{Explicit solution for the Poissonian model}
\label{sec:poi}
In this section we solve the kernel equation~\eqref{kernel_equation_compound_Poisson_model} by establishing an explicit integral expression for the Laplace transform $F_1(s_1)$, see Theorem~\ref{thm:BVPsolution} below, with $F_2(s_2)$ being analogous. Additionally, in Corollary~\ref{cor:totaldomination} we determine $p_1(0,0)$, the probability of total domination starting from the origin, and in Lemma~\ref{lem:F0} we find a simple formula for $F(q_2/r_2,q_1/r_1)$. It would be interesting to understand if this formula can be explained by a direct probabilistic reasoning. We also obtain the asymptotics of $p_1(u,0)$ and $p_1(0,v)$ as $u,v\to\infty$, see Proposition~\ref{prop:asympt}. We adapt the analytic method from~\cite{fayolle_random_2017} which relies on the following steps: study of the kernel $\psi$, analytic continuation of $F_1$ and study of its singularities, formulation of a boundary value problem and its solution.

Without stating it explicitly we assume in the following that our parameters satisfy the conditions of Proposition~\ref{prop:Poisson_kernel}.

\subsection{Study of the kernel}
Consider the kernel $\psi(s_1,s_2)$ given in~\eqref{eq:psi}. The basic idea is to consider its zeros, and so we  define the bi-valued functions $S_1$ and $S_2$ such that
\[
\psi(S_1(s_2),s_2)=0\quad \text{and}\quad \psi(s_1,S_2(s_1))=0.
\]
To do so, we remark that $\psi(s_1,s_2)=0$ is equivalent to
\[
a(s_1)s_2^2+b(s_1)s_2+c(s_1)=0
\]
where
\begin{align*}
a(s_1)&\coloneqq s_1c_2+c_2q_1,\qquad
b(s_1) \coloneqq s_1^2c_1+s_1(c_1q_1+c_2q_2-\lambda_1-\lambda_2)-\lambda_2q_1+c_2q_2q_1,\\
c(s_1) &\coloneqq s_1^2c_1q_2+s_1(-\lambda_1q_2+c_1q_1q_2).
\end{align*}
We also note
\[
d(s_1)\coloneqq b^2(s_1)-4a(s_1)c(s_1)
\]
which is a fourth degree polynomial with roots denoted by $x_1,x_2,x_3,x_4$. Similarly we define $\widetilde{a}$, $\widetilde{b}$, $\widetilde{c}$, $\widetilde{d}$, and let $y_i$ be the four roots of $\widetilde{d}$. Then we have
\[
S_2(s_1)\coloneqq\dfrac{-b(s_1)\pm\sqrt{d(s_1)}}{2a(s_1)}
\quad \text{and}\quad
S_1(s_2)\coloneqq\dfrac{-\widetilde{b}(s_2)\pm\sqrt{\widetilde{d}(s_2)}}{2\widetilde{a}(s_2)}.
\]
The branch points of $S_2$ are the points $x_i$ and the branch points of $S_1$ are the points $y_i$.

\begin{lem}[Branch points]
\label{lem:branchpoint}
The polynomial $d(s_1)$ has four real roots $x_i$ which satisfy
\[
-q_1<x_1<x_2<0<-q_1+\sqrt{\lambda_1q_1/c_1}<x_3<x_4.
\]
The polynomial $d$ is then negative on $[x_1,x_2]\cup [x_3,x_4]$ and positive on $\mbR\setminus ([x_1,x_2]\cup [x_3,x_4])$. The same result hold for the roots $y_i$ of $\widetilde{d}$.
\end{lem}

\begin{proof}
First, remark that for all $s_1\in (-\infty,-q_1]\cup [0,\lambda_1/c_1-q_1]$ we have $-4a(s_1)c(s_1)\ges 0$ and then $d(s_1)>0$ (since the roots of $b$ are different from $-q_1$, $0$, $\lambda_1/c_1-q_1$). For $s_1\in (-q_1,0)\cup (\lambda_1/c_1-q_1,\infty)$ we have $-4a(s_1)c(s_1)<0$. We denote by $x^\pm$ the two roots of $b$ and remark that $-q_1<x^-<0<\lambda_1/c_1-q_1<x^+$, so that $d(x^\pm)=-4a(x^\pm)c(x^\pm)<0$. Additionally, we have $d(s_1)\to +\infty$ as $s_1\to +\infty$. Now we conclude by the intermediate value theorem and noticing that $-q_1+\sqrt{\lambda_1q_1/c_1}<x^+$.
\end{proof}

By Lemma~\ref{lem:branchpoint}, $d(s_1)$ is positive for $s_1\in [x_2,x_3]$ and we can take on this interval the usual square root $d$ without sign ambiguity. We define $\sqrt{d}$ as the analytic function on the cut plane $\mbC\setminus ([x_1,x_2]\cup [x_3,x_4])$ which coincides with the usual square root of $d$ on $[x_2,x_3]$. We denote by $S_2^+$ the branch of the bi-valued function $S_2$ which is equal to $(-b+\sqrt{d})/(2a)$ and which is analytic on $\mbC\setminus ([x_1,x_2]\cup [x_3,x_4])$. We denote by $S_2^-$ the other branch. See Figure~\ref{fig:kernel} and~\ref{fig:zoom} to visualize these functions on $\mbR$. In the same way, we denote by $S_1^+$ and $S_1^-$ the two branches of $S_1$ which are analytic on $\mbC\setminus ([y_1,y_2]\cup [y_3,y_4])$.

\begin{figure}[ht!]
\begin{center}
\includegraphics[scale=0.8]{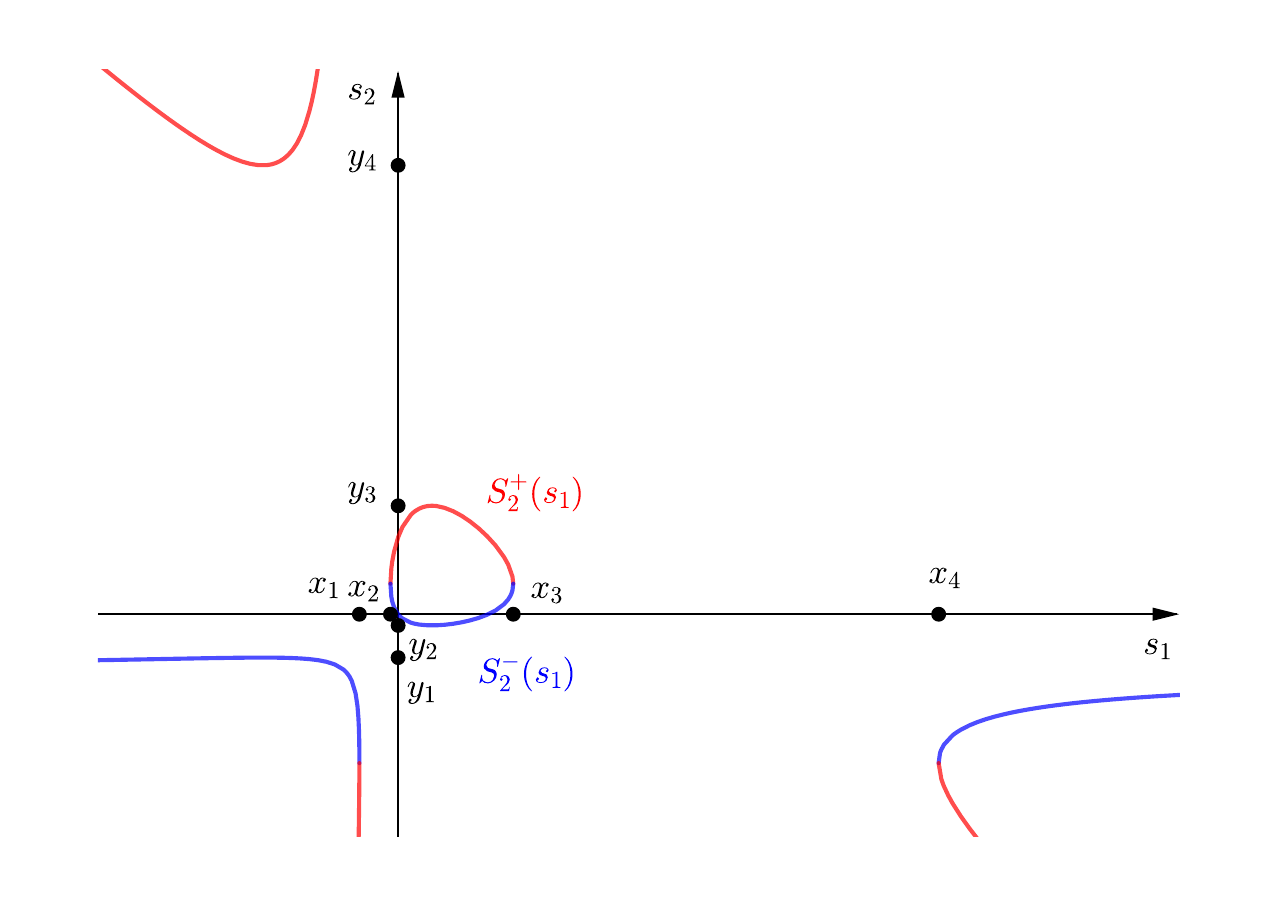}
\vspace{-0.8cm}
\caption{General shape of the curve $\{(s_1,s_2)\in\mbR^2\colon \psi(s_1,s_2)=0\}$ divided in two parts: the function $S_2^-$ (blue) and the function $S_2^+$ (red).}
\label{fig:kernel}
\end{center}
\end{figure}

\begin{figure}[ht!]
\begin{center}
\includegraphics[scale=0.3]{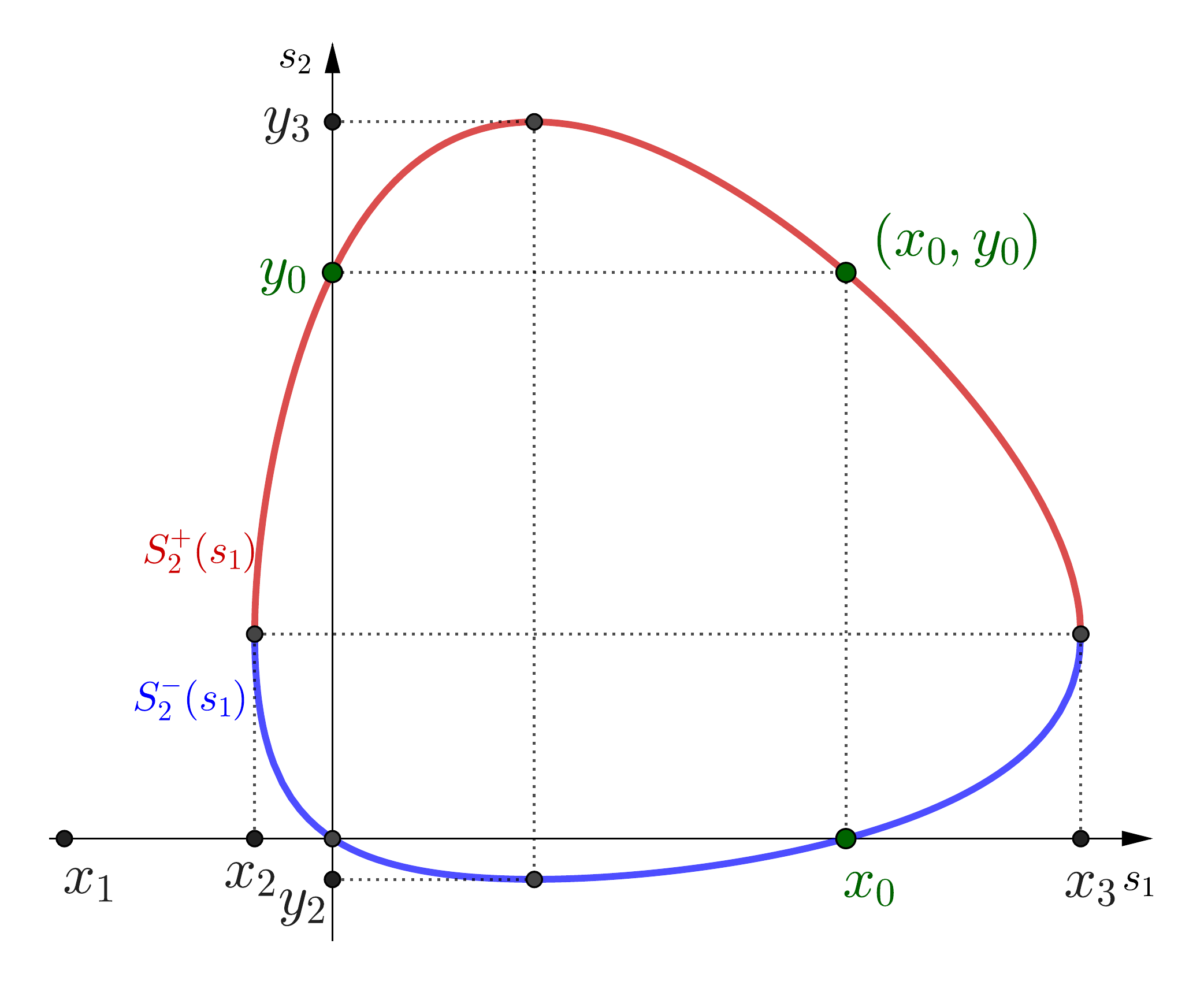}
\vspace{-0.4cm}
\caption{Zoom of Figure~\ref{fig:kernel}: the branch points $x_i$ and $y_i$ are in black, the points $x_0$ and $y_0$ are in green.}
\label{fig:zoom}
\end{center}
\end{figure}

\begin{figure}[ht!]
\begin{center}
\includegraphics[scale=1]{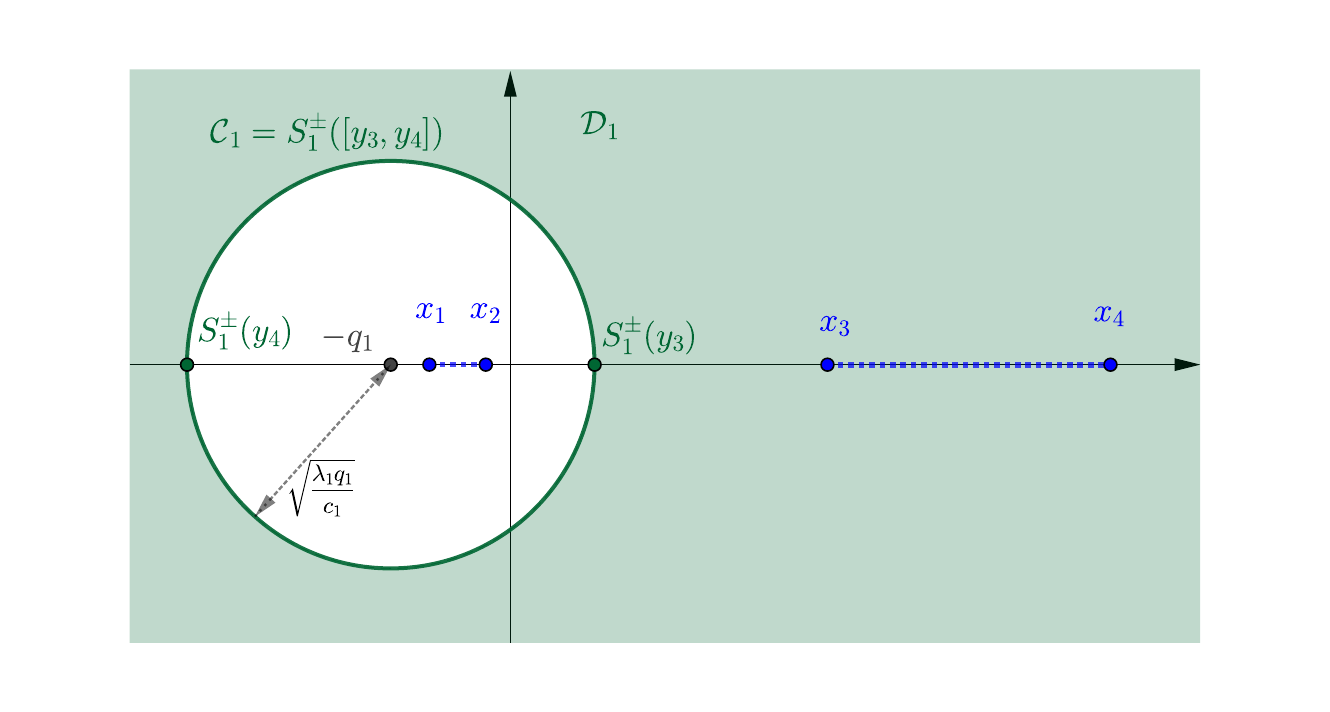}
\caption{Complex plane of the $s_1$ variable: in blue the branch points $x_i$ and the cuts on the complex plane, in green the circle $\mcC_1$ and the domain $\mcD_1$.}
\label{fig:courbes}
\end{center}
\end{figure}

For further use, we define the curve
\[
\mcC_1\coloneqq S_1^\pm([y_3,y_4])=\left\lbrace\dfrac{-\widetilde{b}(y)\pm i\sqrt{-\widetilde{d}(y)}}{2\widetilde{a}(y)}\colon y\in [y_3,y_4]\right\rbrace.
\]
This curve will be the boundary in the boundary value problem established in Section~\ref{sec:BVP}.

\begin{lem}[Circle $\mcC_1$]
The curve $\mcC_1$ is a circle with centre at $-q_1$ and radius $\sqrt{\frac{\lambda_1q_1}{c_1}}$.
\end{lem}

\begin{proof}
By definition, if $s_1\in\mcC_1$ then there exists $s_2\in [y_3,y_4]$ such that $\psi(s_1,s_2)=0$ and we also have $\overline{s_1}\in\mcC_1$ and $\psi(\overline{s_1},s_2)=0$. It implies that $\psi(s_1,s_2)=\psi(\overline{s_1},s_2)$, that is
\[
c_1s_1+\dfrac{\lambda_1q_1}{s_1+q_1}=c_1\overline{s_1}+\dfrac{\lambda_1q_1} {\overline{s_1}+q_1}.
\]
Then we find that
\[
\abs*{s_1+q_1}^2=\dfrac{\lambda_1q_1}{c_1}.
\]
We deduce that $\mcC_1$ is included in the circle of centre $-q_1$ and radius $\sqrt{\frac{\lambda_1q_1}{c_1}}$. Furthermore, as $S_1^+(y_i)=S_1^-(y_i)$ it implies that $\mcC_1$ is a closed curve, which concludes the proof.
\end{proof}

In fact, we may choose the interval $[y_1,y_2]$ instead of $[y_3,y_4]$, since $\mcC_1=S_1^\pm([y_1,y_2])$. Finally, we define the domain
\[
\mcD_1\coloneqq\left\lbrace s_1\in\mbC\colon \abs*{s_1+q_1}^2>\dfrac{\lambda_1q_1}{c_1}\right\rbrace,
\]
which is the complementary of the disc defined by the circle $\mcC_1$, see Figure~\ref{fig:courbes}. We deduce from Lemma~\ref{lem:branchpoint} that $x_3,x_4$ are in $\mcD_1$ and that $x_1,x_2$ are not.

\subsection{Analytic continuation and asymptotics}
\label{sec:continuation}
The goal of this section is to continue analytically $F_1$ to the domain $\mcD_1$ and to study its singularities in order to compute the asymptotics of $p_1(u,0)$ and $p_1(0,v)$, see Proposition~\ref{prop:asympt}.

\begin{lem}[Analytic continuation]
\label{lem:continuation}
The function $F_1(s_1)$ can be meromorphically extended to the set
\[
\{s_1\in\mbC\colon \Re s_1>0 \text{ or } \Re S_2^+(s_1)>0\}
\]
thanks to the formula
\begin{equation}
\label{eq:continuation}
F_1(s_1)=F_1(q_2/r_2)+\dfrac{\psi_2(s_1,S_2^+(s_1)) \left[F_2(q_1/r_1)-F_2(S_2^+(s_1))\right]-F_0}{\psi_1(s_1,S_2^+(s_1))}.
\end{equation}
The analogous result holds for $F_2$.
\end{lem}

\begin{proof}
We are going to use the principle of analytic continuation. The Laplace transforms $F_i(s)$ are analytic on $\{s\in\mbC\colon \Re s >0\}$. According to the kernel equation~\eqref{kernel_equation_compound_Poisson_model}, for $s_1$ and $s_2$ with positive real parts and such that $\psi(s_1,s_2)=0$ we have
\[
0=\psi_1(s_1,s_2)(F_1(s_1)-F_1(q_2/r_2))+\psi_2(s_1,s_1)(F_2(s_2)-F_2(q_1/r_1))+F_0.
\]
When $s_1 \to 0$ for $s_1>0$ we have $S_2^+(s_1)\to\frac{\lambda_2}{c_2}-q_2=y_0>0$. Thus the open connected set
\[
D\colon\{s_1\in\mbC\colon \Re S_2^+(s_1)>0\}
\]
intersects the open set $\{s_1\in\mbC\colon \Re s_1>0\}$. For $s_1$ in this intersection the equation~\eqref{eq:continuation} is satisfied. Then, defining $F(s_1)$ as in~\eqref{eq:continuation} we extend meromorphically $F_1$ to the whole $D$ thanks to the principle of analytic continuation. See Figure~\ref{fig:domainD} representing the domain $D$.
\end{proof}

\begin{lem}[Domain $\mcD_1$]
\label{lem:domain}
The set $\mcD_1$ is included in $\{s_1\in\mbC\colon \Re s_1>0\text{ or } \Re S_2^+(s_1)>0\}$ and $F_1$ is then meromorphic on $\mcD_1$.
\end{lem}

\begin{proof}
It is enough to show that $\mcD_1\cap \{s_1\in\mbC\colon \Re s_1<0\} $ is included in the domain $D$. See Figures~\ref{fig:courbes} and~\ref{fig:domainD} to visualize these sets. By definition if $s_1\in\mcC_1$, we have $S_2^+(s_1)\in [y_3,y_4]$ and then $\Re S_2^+(s_1)>0$. We deduce that the circle $\mcC_1$ is included in $D$. Furthermore, remark that
\[
S_2^+(s_1)\underset{\abs*{s_1}\to\infty}{\sim} -\dfrac{c_1}{c_2}s_1,
\]
which implies that when $s_1$ is large and such that $\Re s_1<0$ we have  $\Re S_2^+(s_1)>0$. The maximum principle applied to the function $S_2^+(s_1)$ implies that $\Re S_2^+(s_1)$ is positive on the set $\mcD_1\cap \{s_1\in\mbC\colon \Re s_1<0\} $. We conclude with Lemma~\ref{lem:continuation}.
\end{proof}

\begin{figure}[ht!]
\begin{center}
\includegraphics[scale=1]{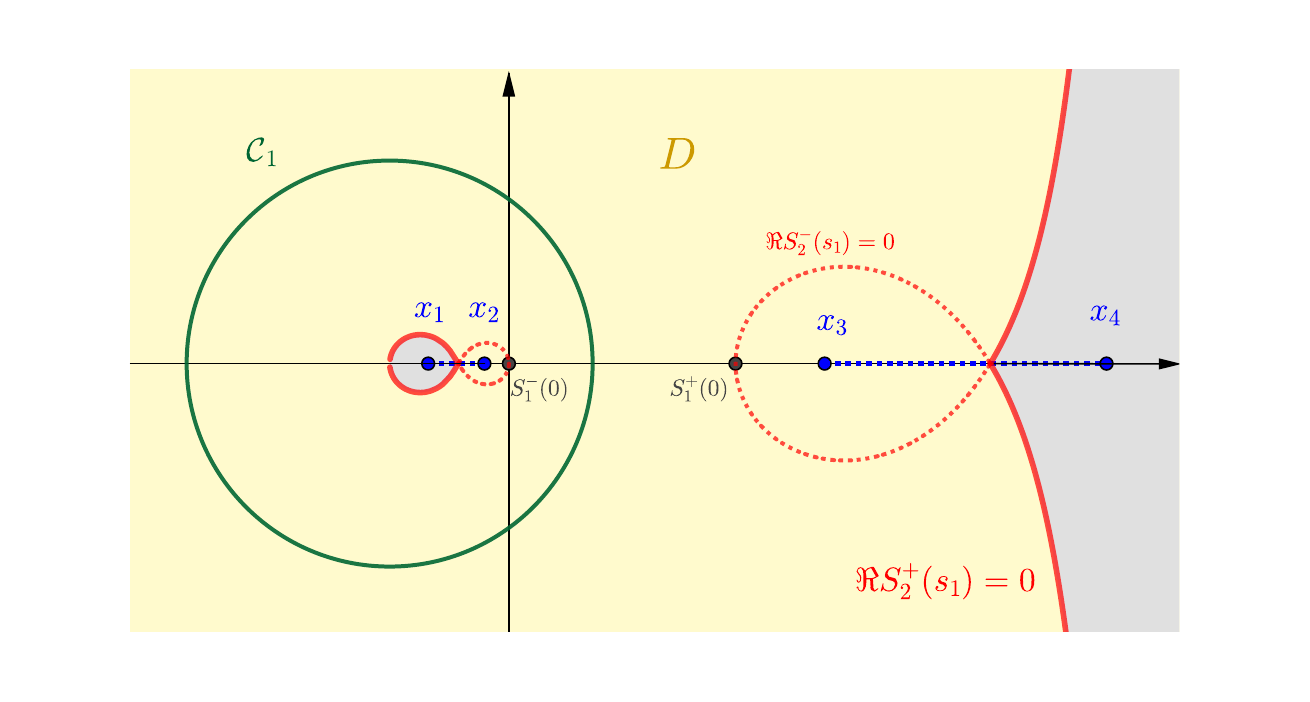}
\vspace{-0.6cm}
\caption{Representation of the $s_1$-complex plane: the domain $D\coloneqq\{s_1\in\mbC\colon \Re S_2^+(s_1)>0\}$ is in yellow, the red curve is the set $\{s_1\in\mbC\colon \Re S_2^+(s_1)=0\}$. The red dotted curve is the set $\{s_1\in\mbC\colon \Re S_2^-(s_1)=0\}$ (note that we do not use this curve).}
\label{fig:domainD}
\end{center}
\end{figure}

Let us recall that $x_2$ and $y_2$ are the roots defined in Lemma~\ref{lem:branchpoint}.

\begin{lem}[Poles of $F_1$ and $F_2$]
\label{lem:poleD}
The polynomial
\begin{equation}
\label{eq:polynomial}
P(s_1)\coloneqq (s_1-q_2/r_2)(s_1+q_1)(r_2c_2-c_1)+\lambda_2(s_1+q_1)+\lambda_1(s_1-q_2/r_2)
\end{equation}
has two real roots $s_1^p\in (-q_1,0)$ and $\widetilde{s}_1^p$ when $r_2c_2-c_1\neq 0$ and one real root $s_1^p\in (-q_1,0)$ when $r_2c_2-c_1=0$.

The meromorphic function $F_1(s_1)$ has at most two poles in $\{s_1\in\mbC\colon \Re s_1>0 \text{ or } \Re S_2^+(s_1)>0\}$ which are $0$ and $s_1^p$:
\begin{itemize}
\item
$0$ is always a simple pole of $F_1$,

\item
$s_1^p$ is a (simple) pole of $F_1$ if and only if $\psi_1(x_2,S_2^\pm(x_2))<0$.
\end{itemize}
Furthermore, $F_1$ has no poles in $\mcD_1$ and is analytic on this set.

In the same way we define $s_2^p\in (-q_2,0)$ which is a (the only) pole of $F_2$ if and only if $\psi_2(S_1^\pm(y_2),y_2)<0$.
\end{lem}

\begin{proof}
The function $F_1$ is initially defined as a Laplace transform which converges on $\{s_1\in\mbC\colon \Re s_1>0\}$. Thus, $F_1$ has no poles on this set. The limits in Theorem~\ref{thm:domination} imply that $\hat{F}_1(0+)=1$ (and $\hat{F}_2(0+)=0$) and we deduce that $0$ is a simple pole of $F_1$ (and that $0$ is not a pole of $F_2$). The analytic continuation of $F_1$ is obtained thanks to formula~\eqref{eq:continuation}. Therefore, the only poles of $F_1$ comes from the $s_1$ of real part negative such that
\[
\psi_1(s_1,S_2^+(s_1))=0.
\]
First, we show that the following system has three solutions: $0$, $s_1^p$ and $\widetilde{s}_1^p$. We have
\begin{align*}
\begin{cases}
\psi(s_1,s_2)=0\\
\psi_1(s_1,s_2)=0
\end{cases}
\Leftrightarrow
\begin{cases}
s_1(c_1-\dfrac{\lambda_1}{q_1+s_1})+s_2(c_2-\dfrac{\lambda_2}{q_2+s_2})=0\\
-\dfrac{\lambda_2}{(q_2+s_2)}=\dfrac{c_2(s_1r_2-q_2)}{q_2}
\end{cases}
\\
\Leftrightarrow
\begin{cases}
s_1\left((c_1-\dfrac{\lambda_1}{q_1+s_1})+s_2\dfrac{c_2 r_2}{q_2}\right)=0\\
s_2=\dfrac{\lambda_2q_2}{c_2(q_2-s_1r_2)}-q_2
\end{cases}
\Leftrightarrow
\begin{cases}
s_1P(s_1)=0\\
s_2=\dfrac{\lambda_2q_2}{c_2(q_2-s_1r_2)}-q_2
\end{cases}
\end{align*}
where $P(s_1)$ is a second degree polynomial defined by~\eqref{eq:polynomial}. Notice that
\[
P(0)=q_1q_2\left(c_1-\dfrac{\lambda_1}{q_1}- r_2\left(c_2-\dfrac{\lambda_2}{q_2}\right)\right)>0
\]
which is positive thanks to assumption~\eqref{A2} (where $\mu_i=c_i-\lambda_i/q_i$) and that
\[
P(-q_1)=-\lambda_1(q_1+q_2/r_2)<0.
\]
We deduce that the two roots of $P$ are real and that one of them that we denote by ${s_1}^p$ satisfy $-q_1<s_1^p<0$ and then $s_1^p\notin\mcD_1$. We have $s_1^p$ is a (simple) pole of $F_1$ if and only if $\psi_1(s_1^p,S_2^+(s_1^p))=0$, i.e. $\psi_1(x_2,S_2^\pm(x_2))<0$, see Figure~\ref{fig:intersec} for a geometric representation. We now show that the second root of $P$ denoted by $\widetilde{s}_1^p$ is not a pole of $F_1$. Firstly, this is clearly the case when $\widetilde{s}_1^p>0$. Secondly, $\widetilde{s}_1^p<-q_1<0$ is not a pole of $F_1$, because we have $\psi_1(\widetilde{s}_1^p,S_2^-(\widetilde{s}_1^p))=0$, but $\psi_1(\widetilde{s}_1^p,S_2^+(\widetilde{s}_1^p))\neq 0$, see Figure~\ref{fig:intersec} for a geometric representation.
\end{proof}

\begin{figure}[ht!]
\begin{center}
\includegraphics[scale=0.24]{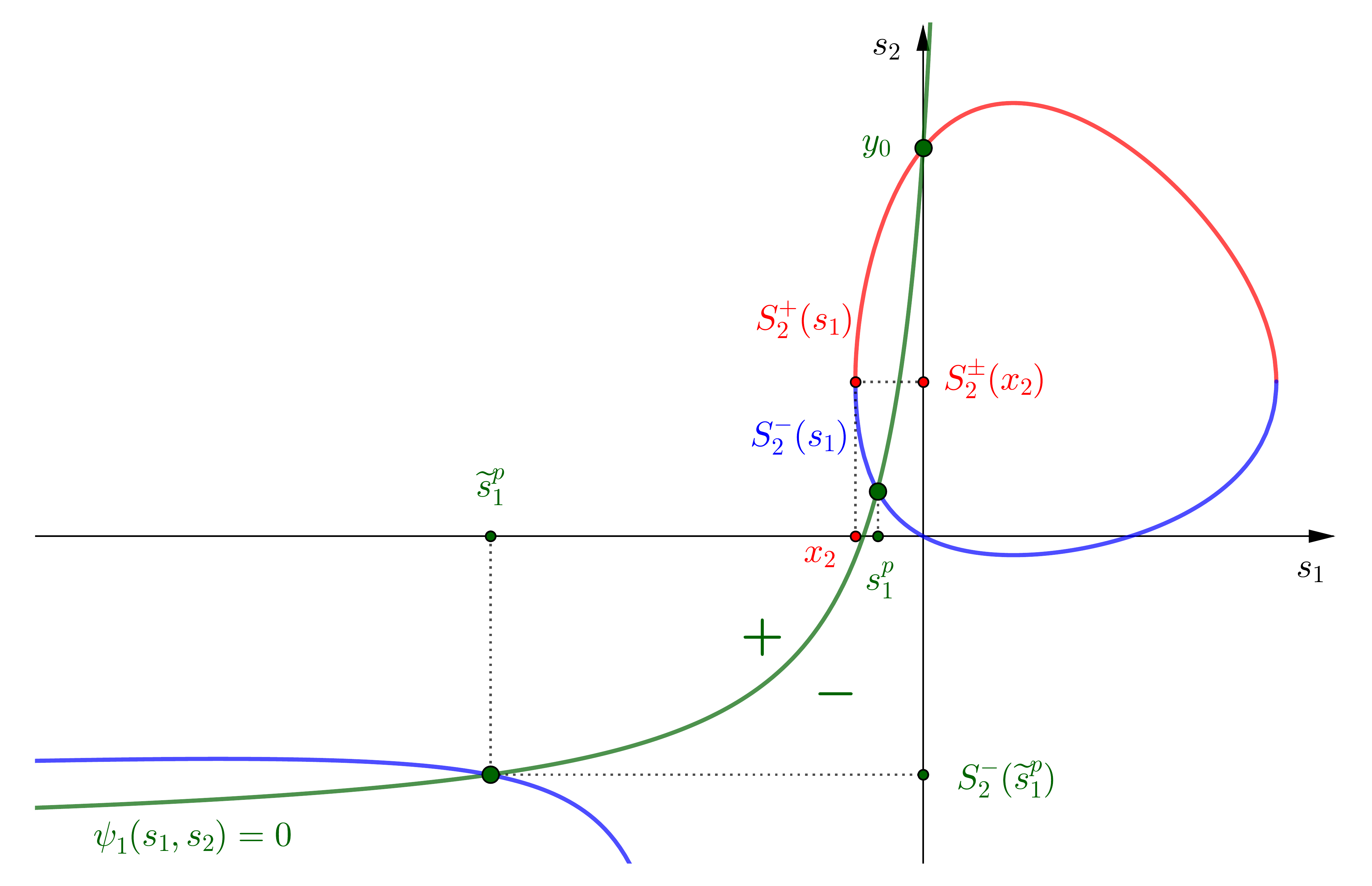}
\vspace{-0.4cm}
\caption{In green the curve $\psi_1(s_1,s_2)=0$ and its intersections with the curve $\psi(s_1,s_2)=0$. In this case $\psi_1(x_2,S_2^\pm(x_2))>0$ and then $s_1^p$ is not a pole of $F_1$.}
\label{fig:intersec}
\end{center}
\end{figure}

Our next result establish the rate of decay of $p_2(u,0)=1-p_1(u,0)$. It is noted that the analogous result holds true for $p_1(0,v)$ as $v\to\infty$.

\begin{prop}[Asymptotics of domination]
\label{prop:asympt}
The asymptotic behaviour of $p_1(u,0)$ as $u\to\infty$ is given by
\[
1-p_1(u,0)\sim C
\begin{cases}
e^{us_1^p} & \text{ if } \psi_1(x_2,S_2^\pm(x_2))<0,\\
u^{-\frac{3}{2}}e^{u x_2} & \text{ if } \psi_1(x_2,S_2^\pm(x_2))>0,\\
u^{-\frac{1}{2}}e^{u x_2} & \text{ if } \psi_1(x_2,S_2^\pm(x_2))=0,
\end{cases}
\]
for some constant $C$ which depends on the case, where ${s}_1^p$ is defined in Lemma~\ref{lem:poleD}.
\end{prop}

\begin{proof}
The asymptotics of a function derives from the largest singularity of its Laplace transform, see for example~\cite[Theorem~37.1]{doetsch_introduction_1974}. Assume that $f(u)$ is a function, $L(s)$ is its Laplace transform, and $a$ is the largest singularity of order $k$ (i.e. in the neighbourhood of $a$ the Laplace transform $F$ behaves as $(s-a)^{-k}$ up to additive and multiplicative constants). Then apply the theorems stating that $f(u)$ is equivalent to $u^{k-1}e^{au}$ up to a constant as $u\to\infty$.

The Laplace transform of interest is $1/s_1-F_1(s_1)$. By Lemma~\ref{lem:poleD} the point $0$ is not a singularity, whereas ${s}_1^p$ is a simple pole and the largest singularity of $F_1$ if $\psi_1(x_2,S_2^\pm(x_2))<0$. In that case the asymptotics is then given by $Ce^{us_1^p}$ for some constant $C$. When $\psi_1(x_2,S_2^\pm(x_2))\ges 0$, the largest singularity is the branch point $x_2$. Thanks to the definition of $S_2^+$ and the analytic continuation formula~\eqref{eq:continuation} we obtain for some constants $C_i$ that
\[
F_1(s_1)\underset{s_1\to x_2}{=}
\begin{cases}
C_1+C_2\sqrt{s_1-x_2}+\Oh(s_1-x_2) & \text{ if } \psi_1(x_2,S_2^\pm(x_2))>0,\\
\dfrac{C_3}{\sqrt{s_1-x_2}}+\Oh(1) & \text{ if } \psi_1(x_2,S_2^\pm(x_2))=0.
\end{cases}
\]
The result now follows.
\end{proof}

\subsection{Boundary value problem and its solution}
\label{sec:BVP}
We are now ready to establish a boundary value problem (BVP) satisfied by $f_1$ defined in~\eqref{eq:F0tilde}. It is a Carleman homogeneous BVP which relies on the domain $\mcD_1$ and the boundary $\mcC_1$.

\begin{prop}[BVP]
\label{prop:BVP}
The function $f_1$ satisfies the following Carleman boundary value problem:
\begin{enumerate}[label=(\roman*)]
\item
\label{item:1}
$f_1(s_1)$ is analytic on $\mcD_1$;

\item
\label{item:2}
$\lim_{s_1\to\infty} f_1(s_1)=\dfrac{F_0}{\widetilde{F}_0}\dfrac{r_2}{q_2}-F_1(q_2/r_2)$;

\item
\label{item:3}
$f_1$ satisfies the boundary condition
\[
f_1(\overline{s_1})=G(s_1)f_1(s_1),\quad \forall\; s_1\in\mcC_1,
\]
where
\begin{equation}
\label{eq:Gdef}
G(s_1)\coloneqq\dfrac{\psi_1}{\psi_2}(s_1,S_2^+(s_1)) \dfrac{\psi_2}{\psi_1}(\overline{s_1},S_2^+(s_1)).
\end{equation}
\end{enumerate}
\end{prop}

\begin{proof}
Item~\ref{item:1} directly derives from Lemma~\ref{lem:domain} and Lemma~\ref{lem:poleD}. Item~\ref{item:2} comes from the fact that the Laplace transform $F_1$ converges to $0$ at infinity. Item~\ref{item:3} comes from the kernel equation~\eqref{eq:kerneleqsimple}. For $s_1\in\mcC_1$, we have $\overline{s_1}\in\mcC_1$ and $S_2^+(s_1)=S_2^+(\overline{s_1})$. We evaluate~\eqref{eq:kerneleqsimple} at $(s_1,S_2^+(s_1))$ and $(\overline{s_1},S_2^+(s_1))$ and we obtain the two equations
\[
\begin{cases}
0=\psi_1(s_1,S_2^+(s_1))f_1(s_1)+\psi_2(s_1,S_2^+(s_1))f_2(S_2^+(s_1)),\\
0=\psi_1(\overline{s_1},S_2^+(s_1))f_1(\overline{s_1})+ \psi_2(\overline{s_1},S_2^+(s_1))f_2(S_2^+(s_1)).
\end{cases}
\]
Eliminate $f_2(S_2^+(s_1))$ from these equations gives the boundary condition~\ref{item:3}.
\end{proof}

To solve the boundary value problem on $\mcD_1$ we need to introduce a conformal function which glues together the upper part and the lower part of the circle $\mcC_1$. This gluing function is a simple rational function and derives from the kernel. See Figure~\ref{fig:w} to visualize the gluing function.

\begin{lem}[Conformal gluing function $w$]
\label{lem:gluingw}
The function
\begin{equation}
\label{eq:wdef}
w(s_1)\coloneqq\dfrac{1}{2}\left(\dfrac{s_1+q_1}{\sqrt{\lambda_1q_1/c_1}}+ \dfrac{\sqrt{\lambda_1q_1/c_1}}{s_1+q_1}\right)
\end{equation}
satisfies the following properties
\begin{enumerate}[label=(\roman*)]
\item
$w$ is holomorphic in $\mcD_1$ and continuous on $\overline{\mcD_1}$;

\item
$w$ is one-to-one from $\mcD_1$ to $\mbC\setminus [-1,1]$;

\item
$w$ satisfies the boundary property
\[
w(s_1)=w(\overline{s_1}),\quad \forall\, s_1\in\mcC_1.
\]
\end{enumerate}
\end{lem}

\begin{proof}
Recall that $s_1\in\mcC_1$ if and only if $\abs*{s_1+q_1}^2=\frac{\lambda_1q_1}{c_1}$. The three items are derived by means of straightforward calculus.
\end{proof}

We write $\mcC_1^-$ (resp. $\mcC_1^+$) for the half circle defined by the intersection of $\mcC_1$ and the half plane of negative (resp. positive) imaginary part, see Figure~\ref{fig:w}. The circle $\mcC_1$ and the half circles $\mcC_1^\pm$ are counterclockwise oriented.

\begin{figure}[ht!]
\begin{center}
\includegraphics[scale=0.8]{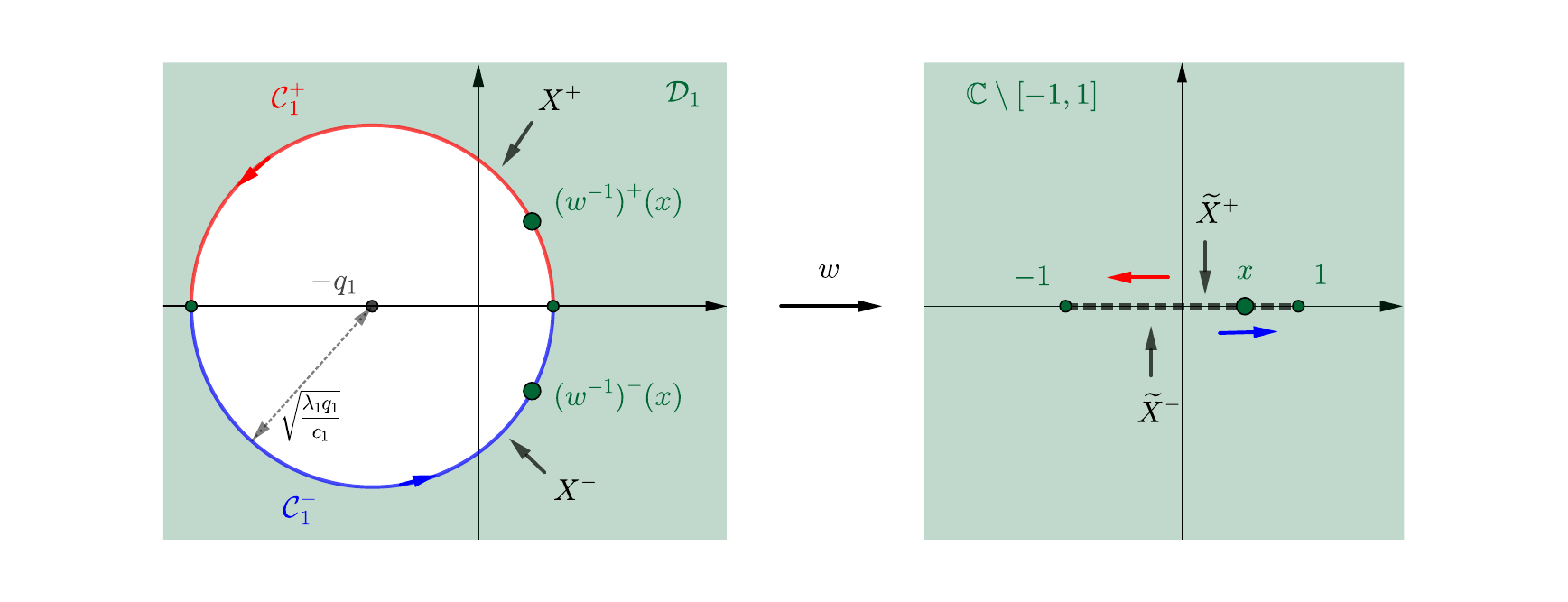}
\caption{Conformal gluing function $w$ is one to one from $\mcD_1$ to $\mcC\setminus [-1,1]$.}
\label{fig:w}
\end{center}
\end{figure}

To solve the BVP we need to compute the index which is defined by
\[
\chi\coloneqq\dfrac{1}{2\pi}[\arg G(s_1)]_{\mcC_1^-}=\dfrac{1}{2\pi} \left[\arg\dfrac{\psi_1}{\psi_2}(s_1,S_2^+(s_1))\right]_{\mcC_1}.
\]
The index represents the variation of the argument of $G(s_1)$ when $s_1$ lies on the half circle $\mcC_1^-$, that is the difference between initial and final value when the argument varies continuously along the half circle. The second equality comes from the definition of $G$ in~\eqref{eq:Gdef}. Thus, equivalently, it is also the variation of the argument of $\psi_1/\psi_2$ around the circle $\mcC_1$.

\begin{lem}[Index]
\label{lem:chipole}
The index $\chi$ given by
\begin{equation}
\label{eq:chidef}
\chi=
\begin{cases}
0\text{ if } q_2/r_2\les -q_1+\sqrt{\lambda_1q_1/c_1}\Leftrightarrow\ f_1\text{ has no zeros in $\mcD_1$,}\\
1\text{ if } q_2/r_2>-q_1+\sqrt{\lambda_1q_1/c_1}\Leftrightarrow\ f_1 \text{ has one zero ($q_2/r_2$) in $\mcD_1$.}
\end{cases}
\end{equation}
\end{lem}

\begin{proof}
Consider the curve $\frac{\psi_1}{\psi_2}(s_1,S_2^+(s_1))$ when $s_1$ lies on $\mcC_1$. This curve is numerically plotted in Figure~\ref{fig:index} in both cases of interest. Let us denote $A=\frac{\psi_1}{\psi_2}(-q_1-\sqrt{\lambda_1q_1/c_1},y_4)$ and $B=\frac{\psi_1}{\psi_2}(-q_1+\sqrt{\lambda_1q_1/c_1},y_3)$ the image by $\frac{\psi_1}{\psi_2}$ of the two real points of $\mcC_1$. Analysis of the equation defining this curve shows also that there is another double real point that we denote by~$C$.

We can show that $A$ and $C$ are always positive. On the other hand $B<0$ if and only if $q_2/r_2>-q_1+\sqrt{\lambda_1q_1/c_1}$. The last property comes from the fact that the line $s_1=q_2/r_2$ is the asymptote of the hyperbola $\psi_1(s_1,s_2)=0$ and the position of the point $(-q_1+\sqrt{\lambda_1q_1/c_1},y_3)$ w.r.t. this asymptote determines the sign of $B$. Now we see that when $q_2/r_2>-q_1+\sqrt{\lambda_1q_1/c_1}$ the curve of interest make a positive turn around the origin, i.e. $\chi=1$. In the other case, $B>0$ and the curve makes no turns around the origin, i.e. $\chi=0$.

\begin{figure}[ht!]
\begin{center}
\includegraphics[scale=0.95]{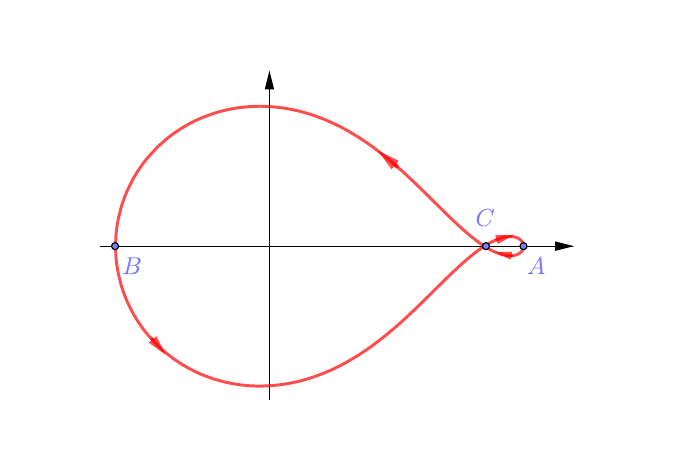}
\includegraphics[scale=0.95]{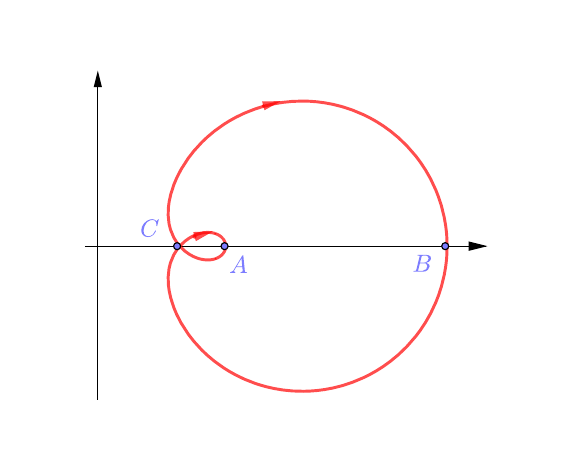}
\vspace{-0.8cm}
\caption{Plot of $\frac{\psi_1}{\psi_2}(s_1,S_2^+(s_1))$ when $s_1$ lies on $\mcC_1$. Left: $q_2/r_2\in\mcD_1$ and $\chi=1$; right: $q_2/r_2\notin\mcD_1$ and $\chi=0$.}
\label{fig:index}
\end{center}
\end{figure}

Alternatively, one may start by noticing that by the boundary condition of Proposition~\ref{prop:BVP}
\begin{align*}
\chi=\dfrac{1}{2\pi}\left[\arg\dfrac{f_1(\overline{s_1})}{f_1(s_1)}\right]_{\mcC_1^-}
=\dfrac{-1}{2\pi}\left[\arg{f_1(s_1)}\right]_{\mcC_1}=Z_{\mcD_1}(f_1)-P_{\mcD_1}(f_1),
\end{align*}
where $Z_{\mcD_1}(f_1)$ is the number of zeros (counted with multiplicity) of the meromorphic function $f_1$ in $\mcD_1\cup \{\infty\}$ and $P_{\mcD_1}(f_1)$ is the number of poles (counted with multiplicity) of $f_1$ in $\mcD_1\cup \{\infty\}$. By Lemma~\ref{lem:poleD} function $f_1$ has no poles in $\mcD_1\cup \{\infty\}$, so that $\chi\ges 0$ and it is left to analyse the zeros of~$f_1$ remembering that $f_1(q_2/r_2)=0$.
\end{proof}

We are now ready to present an explicit integral expression for $F_1$. The analogous result holds for $F_2$ and thus we obtain an explicit expression for $F$ via the kernel equation. Recall that $G$ is defined in equation~\eqref{eq:Gdef}, $w$ in~\eqref{eq:wdef}, $F_0$ in~\eqref{eq:C0}, $\widetilde{F}_0$ in~\eqref{eq:F0tilde} and $\chi$ is given in~\eqref{eq:chidef}.

\begin{thm}[Explicit expression for $F_1$]
\label{thm:BVPsolution}
The Laplace transform $F_1$ is given by
\begin{equation}
\label{eq:explicitexpression}
F_1(s_1)=\dfrac{F_0}{\widetilde{F}_0}\dfrac{1}{s_1}+\left(\dfrac{F_0}{\widetilde{F}_0} \dfrac{r_2}{q_2}-F_1(q_2/r_2)\right)\left(X(s_1)-1\right),\quad \forall\, s_1\in\mcD_1,
\end{equation}
where
\begin{equation}
\label{eq:X}
X(s_1)\coloneqq\left(\dfrac{w(s_1)-w(q_2/r_2)}{w(s_1)-1}\right)^\chi \exp\left(\dfrac{1}{2i\pi}\int_{\mcC_1^-} \log(G(t))\dfrac{w'(t)}{w(t)-w(s_1)}\D t\right)
\end{equation}
and
\[
F_1(q_2/r_2)=\dfrac{F_0}{\widetilde{F}_0}\dfrac{r_2}{q_2}+\dfrac{F_0}{X(x_0)} \left(\dfrac{1}{\widetilde{F}_0}\left(\dfrac{1}{x_0}-\dfrac{r_2}{q_2}\right)+ \dfrac{1}{\psi_1(x_0,0)}\right).
\]
\end{thm}

Let us provide some comments. Firstly, the given expression is valid for real $s_1$ larger than $\sqrt{\lambda_1q_1/c_1}-q_1>0$. Secondly, we may replace the integral on the half circle of $\log G$ by the integral on the whole circle of $\log\frac{\psi_1}{\psi_2}$, since
\[
\int_{\mcC_1^-} \log(G(t))\dfrac{w'(t)}{w(t)-w(s_1)}\D t=\int_{\mcC_1} \log\left(\dfrac{\psi_1}{\psi_2}(t,S_2^+(t))\right)\dfrac{w'(t)}{w(t)-w(s_1)}\D t.
\]

This theorem establishes the existence of the unique solution of the kernel equation under the limit conditions found in Theorem~\ref{thm:domination}. The uniqueness derives from the solution of the boundary value problem and the value of the index. The same remark can be made about Theorem~\ref{thm:explicitF1Brown}.

\begin{proof}[Proof of Theorem~\ref{thm:BVPsolution}]
To solve the Carleman BVP of Proposition~\ref{prop:BVP} we are going to transform it into a Riemann BVP using the conformal gluing function $w$. See, for example,~\cite[\S5.2]{fayolle_random_2017} which present briefly the main results of BVP theory. We consider the function
\[
\widetilde{f}_1(x)\coloneqq (x-w(q_2/r_2))^{-\chi} f_1\circ w^{-1}(x).
\]
According to Proposition~\ref{prop:BVP}, Lemma~\ref{lem:gluingw}, and the fact that $f_1(q_2/r_2)=0$ we have
\begin{enumerate}[label=(\roman*)]
\item
\label{item:11}
$\widetilde{f}_1$ is analytic on $\mbC\setminus [-1,1]$;

\item
\label{item:22}
$\widetilde{f}_1(x)\underset{\infty}{\sim} x^{-\chi}\left(\frac{F_0}{\widetilde{F}_0} \frac{r_2}{q_2}-F_1(q_2/r_2)\right)$;

\item
\label{item:33}
$\widetilde{f}_1$ has left limits $\widetilde{f}_1^+$ and right limits $\widetilde{f}_1^-$ on $[-1,1]$ which satisfy the boundary condition
\[
\widetilde{f}_1^+(x)=\widetilde{G}(x)\widetilde{f}_1^-(x)
\]
with $\widetilde{G}(x)\coloneqq G((w^{-1})^-(x))$ where we denote by $(w^{-1})^-$ the right limit on $[-1,1]$ of $w^{-1}$, see Figure~\ref{fig:w}.
\end{enumerate}
The function
\[
\widetilde{X}(x)\coloneqq\left({x-1}\right)^{-\chi}\exp\left(\dfrac{1}{2i\pi}\int_{-1}^1 \dfrac{\log\widetilde{G}(u)}{u-x}\D u\right),\quad \forall\, x\notin\mbC\setminus [0,1],
\]
satisfies the homogeneous problem
\[
\widetilde{X}^+(x)=\widetilde{G}(x)\widetilde{X}^-(x),\quad \forall\, x\in [0,1],
\]
where we write $\widetilde{X}^+$ (resp. $\widetilde{X}^-$) for the right (resp. left) limit of $\widetilde{X}$ on $[-1,1]$. This is a classical result of BVP theory stemming from the Sokhotsky--Plemelj formulas, see~\cite[(5.2.24) and Theorem~5.2.3]{fayolle_random_2017}. We deduce from~\ref{item:33} that
\[
\dfrac{\widetilde{f}_1^+}{\widetilde{X}^+}(x)= \dfrac{\widetilde{f}_1^-}{\widetilde{X}^-}(x),\quad \forall\, x\in [0,1].
\]
From~\ref{item:11} it follows that $\frac{\widetilde{f}_1}{\widetilde{X}} $ is analytic in the whole $\mbC$. Thanks to item~\ref{item:22} and to the fact that $\widetilde{X}(x)\sim_\infty x^{-\chi}$ (by Lemma~\ref{lem:chipole} and since the integral in the exponential goes to $0$ when $x$ goes to infinity) we  find that the analytic function $\frac{\widetilde{f}_1}{\widetilde{X}}$ converges to $\frac{F_0}{\widetilde{F}_0} \frac{r_2}{q_2}-F_1(q_2/r_2)$ at infinity. Thus it coincides with this constant, and so
\begin{gather*}
f_1(s_1)=\left(\dfrac{F_0}{\widetilde{F}_0}\dfrac{r_2}{q_2}-F_1(q_2/r_2)\right) (w(s_1)-w(q_2/r_2))^\chi\widetilde{X}(w(s_1))=\\
=\left(\dfrac{F_0}{\widetilde{F}_0}\dfrac{r_2}{q_2}-F_1(q_2/r_2)\right)X(s_1),
\end{gather*}
where the last equality follows by change of variable $u=w(t)$. Now \eqref{eq:explicitexpression} follows from the definition of $f_1$ in~\eqref{eq:F0tilde}.

We now compute the constant $F_1(q_2/r_2)$. Equation~\eqref{eq:valueC} gives
\[
F_1(x_0)-F_1(q_2/r_2)=-\dfrac{F_0}{\psi_1(x_0,0)},
\]
whereas \eqref{eq:explicitexpression} implies that
\[
F_1(x_0)-F_1(q_2/r_2)=\dfrac{F_0}{\widetilde{F}_0} \left(\dfrac{1}{x_0}-\dfrac{r_2}{q_2}\right)+ \left(\dfrac{F_0}{\widetilde{F}_0}\dfrac{r_2}{q_2}-F_1(q_2/r_2)\right)X(x_0),
\]
which readily yield the stated expression for $F_1(q_2/r_2)$.
\end{proof}

We conclude by providing an expression for the probability of total domination when starting from the origin.

\begin{cor}
\label{cor:totaldomination}
The probability of total domination when stating from the origin is given by
\begin{equation}
\label{eq:p1Poisson}
p_1(0,0)=\dfrac{F_0}{\widetilde{F}_0}-\left(\dfrac{F_0}{\widetilde{F}_0}\dfrac{r_2}{q_2}- F_1(q_2/r_2)\right) \dfrac{\sqrt{\lambda_1q_1/c_1}}{i\pi}\int_{\mcC_1^-} \log(G(t))w'(t)\D t.
\end{equation}
\end{cor}

\begin{proof}
We deduce from Theorem~\ref{thm:BVPsolution} that
\[
p_1(0,0)=\lim_{s_1\to\infty} s_1F_1(s_1)=\dfrac{F_0}{\widetilde{F}_0}+ \left(\dfrac{F_0}{\widetilde{F}_0}\dfrac{r_2}{q_2}-F_1(q_2/r_2)\right)\lim_{s_1\to\infty} s_1(X(s_1)-1).
\]
Let us notice that when $s_1\to\infty$ the integral in the exponential of equation~\eqref{eq:X} is equivalent to $C/s_1$ where
\[
C\coloneqq -\dfrac{\sqrt{\lambda_1q_1/c_1}}{i\pi}\int_{\mcC_1^-} \log(G(t))w'(t)\D t.
\]
By Taylor's expansion of $X$ we obtain $X(s_1)=1+C/s_1+\oh(1/s_1)$ and the result follows.
\end{proof}

\section{Explicit solution for the Brownian model}
\label{sec:BM}
In this section we solve the kernel equation~\eqref{Brownian_kernel_equation} for the correlated Brownian model. We obtain an explicit integral expression for $F_1$ and the probability $p_1(0,0)$ in Theorem~\ref{thm:explicitF1Brown}. The asymptotics of $p_1(u,0),u\to\infty$ is given in Proposition~\ref{prop:asymptBrown}. We follow the same steps as in the Poissonian model studied in~\S\ref{sec:poi} and, consequently, some details will be omitted. Importantly, the kernel $\psi$ is similar to the one studied in~\cite{franceschi_2019} and~\cite{baccelli_analysis_1987}, and so its various properties can be taken from there.

Without stating it explicitly we assume in the following that our parameters satisfy the conditions of Proposition~\ref{prop:Brownian_kernel}. In particular, correlation is non-negative $\rho\in [0,1)$. We stress, however, that the main parts of the following analysis can be carried out also for $\rho<0$, and so the remaining hurdle is showing that the same kernel equation holds in this case as well.

\subsection{Study of the kernel}
Reconsider the kernel in~\eqref{Brownian_kernel_equation}, and  define the bi-valued functions $S_1$ and $S_2$ such that
\[
\psi(S_1(s_2),s_2)=0\quad \text{and}\quad \psi(s_1,S_2(s_1))=0.
\]
A direct calculus yields the branches
\[
\begin{cases}
S_1^\pm(s_2)=\dfrac{-(\rho\sigma_1\sigma_2s_2+\mu_1)\pm \sqrt{s_2^2\sigma_1^2\sigma_2^2(\rho^2-1)+ 2s_2\sigma_1(\mu_1\rho\sigma_2-\mu_2\sigma_1)+\mu_1^2}}{\sigma_1^2},\\
S_2^\pm(s_1)=\dfrac{-(\rho\sigma_1\sigma_2s_1+\mu_2)\pm \sqrt{s_1^2\sigma_1^2\sigma_2^2(\rho^2-1)+ 2s_1\sigma_2(\mu_2\rho\sigma_1-\mu_1\sigma_2)+\mu_2^2}}{\sigma_2^2}.
\end{cases}
\]
The respective branch points of $S_1$ and $S_2$ are
\[
\begin{cases}
y^\pm=\dfrac{\mu_1\rho\sigma_1\sigma_2-\mu_2\sigma_1^2\pm \sqrt{(\mu_1\rho\sigma_1\sigma_2-\mu_2\sigma_1^2)^2+\mu_1^2\sigma_1^2\sigma_2^2(1-\rho^2)}} {\sigma_1^2\sigma_2^2(1-\rho^2)},\\
x^\pm=\dfrac{\mu_2\rho\sigma_1\sigma_2-\mu_1\sigma_2^2\pm \sqrt{(\mu_2\rho\sigma_1\sigma_2-\mu_1\sigma_2^2)^2+\mu_2^2\sigma_1^2\sigma_2^2(1-\rho^2)}} {\sigma_1^2\sigma_2^2(1-\rho^2)} .
\end{cases}
\]
The functions $S_1^\pm$ (resp. $S_2^\pm$) are analytic on the cut plane $\mbC\setminus ((-\infty,y^-]\cup [y^+,\infty))$ (resp. $\mbC\setminus ((-\infty,x^-]\cup [x^+,\infty))$). See Figure~\ref{fig:kernelBrownian} to visualize $S_2^\pm$ on $[x^-,x^+]$.

Recall the definition of $x_0,y_0$ in~\eqref{eq:x0y0Brownian}. Furthermore, we define $s_1^p$, the first coordinate of the other intersection between the ellipse $\psi=0$ and the line $\psi_1=0$. Symmetrically we define $s_2^p$. We have
\begin{equation}
\label{eq:s1pBrown}
s_1^p\coloneqq -\dfrac{2(r_2\abs*{\mu_2}-\abs*{\mu_1})} {\sigma_1^2+\sigma_2^2r_2^2-2\rho\sigma_1\sigma_2r_2}<0
\quad \text{and}\quad
s_2^p\coloneqq -\dfrac{2(r_1\abs*{\mu_1}-\abs*{\mu_2})} {\sigma_2^2+\sigma_1^2r_1^2-2\rho\sigma_1\sigma_2r_1}<0.
\end{equation}
See Figure~\ref{fig:kernelBrownian} for a geometric interpretation of $x_0$, $y_0$ and $s_1^p$.

\begin{figure}[ht!]
\begin{center}
\includegraphics[scale=1]{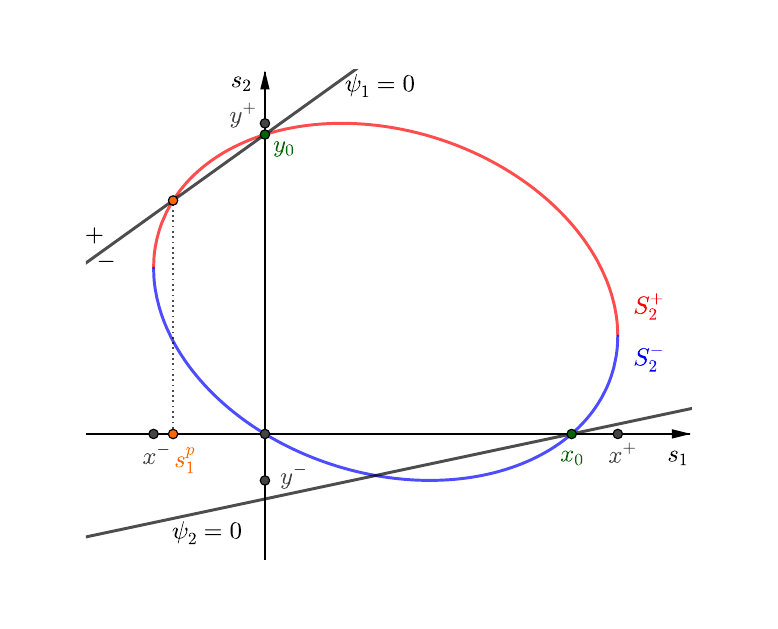}
\vspace{-0.8cm}
\caption{The set $\{(s_1,s_2)\in\mbR^2\colon \psi(s_1,s_2)=0\}$ is an ellipse divided in two parts: in blue the function $S_2^-$ and in red the function $S_2^+$. The two lines are the sets defined by $\psi_1=0$ and $\psi_2=0$. The branch points $x^\pm$ and $y^\pm$ are in black, the points $x_0$ and $y_0$ in green and the pole $s_1^p$ in orange.}
\label{fig:kernelBrownian}
\end{center}
\end{figure}

\begin{figure}[ht!]
\begin{center}
\begin{subfigure}{0.32\textwidth}
\includegraphics[scale=1]{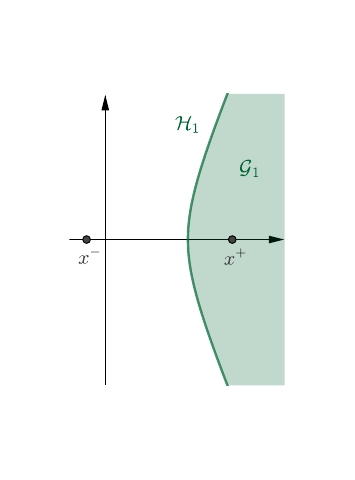}
\caption{$\rho<0$}
\label{fig:gull2}
\end{subfigure}
\begin{subfigure}{0.32\textwidth}
\includegraphics[scale=1]{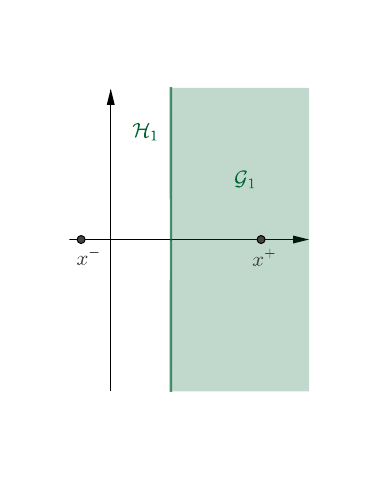}
\caption{$\rho=0$}
\label{fig:tiger2}
\end{subfigure}
\begin{subfigure}{0.32\textwidth}
\includegraphics[scale=1]{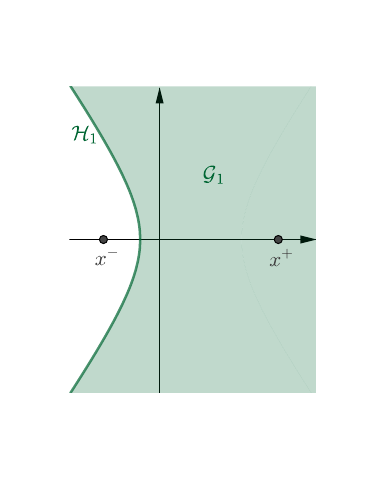}
\caption{$\rho>0$}
\label{fig:mouse2}
\end{subfigure}
\caption{Complex plane of the $s_1$ variable: in green the hyperbola $\mcH_1$ and the domain $\mcG_1$.}
\label{fig:courbesBrownian}
\end{center}
\end{figure}

We now define the curve
\[
\mcH_1\coloneqq S_1^\pm([y^+,\infty))=\{s_1\in\mbC\colon \psi(s_1,s_2)=0 \text{ and } s_2\in [y^+,\infty)\}.
\]
This curve is the boundary of the boundary value problem established in Section~\ref{sec:BVPBrownian}.

\begin{lem}[Hyperbola $\mcH_1$]
\label{lem:hyperbole}
The curve $\mcH_1$ is a branch of hyperbola symmetrical w.r.t. the horizontal axis, whose equation is
\begin{equation*}
\label{eq:hyperbole}
\sigma_1^2\sigma_2^2(\rho^2-1)x^2+\rho^2\sigma_1^2\sigma_2^2y^2- 2(\sigma_2^2\mu_1-\rho\sigma_1\sigma_2\mu_2)x= \mu_1(\sigma_2^2\mu_1-2\rho\sigma_1\sigma_2\mu_2)/\sigma_1^2.
\end{equation*}
The curve $\mcH_1$ is the right branch of the hyperbola if $\rho<0$, the left branch if $\rho>0$, and a straight line when $\rho=0$, see Figure~\ref{fig:courbesBrownian}.
\end{lem}

\begin{proof}
See~\cite[Lemma~4]{franceschi_2019} or~\cite[Lemma~9]{baccelli_analysis_1987} which study a similar kernel and derive the equation of the hyperbola.
\end{proof}

We denote by $\mcH_1^-$ the part of $\mcH_1$ of imaginary part negative. Finally we define the domain $\mcG_1$ which is bounded by $\mcH_1$ and contain $x^+$ (and not $x^-$), see Figure~\ref{fig:courbesBrownian}.

\subsection{Asymptotics results}
Similarly to Section~\ref{sec:continuation} we continue meromorphically $f_1$ and we study its poles in order to compute the asymptotics of $p_1(u,0)$ and $p_1(0,v)$ when $u$ and $v\to\infty$.

\begin{lem}[Analytic continuation]
\label{lem:continuationBrown}
The function $F_1(s_1)$ can be meromorphically extended to the set
\begin{equation}
\label{eq:setBrownian}
\{s_1\in\mbC\colon\Re s_1>0 \text{ or } \Re S_2^+(s_1)>0\}
\end{equation}
thanks to the formula
\begin{equation}
\label{eq:continuationBrown}
F_1(s_1)=\dfrac{-\psi_2(s_1,S_2^+(s_1))F_2(S_2^+(s_1))-cp_1(0,0)}{\psi_1(s_1,S_2^+(s_1))}.
\end{equation}
The domain $\mcG_1$ is included in the set defined in~\eqref{eq:setBrownian} and $F_1$ is then meromorphic on $\mcG_1$.
\end{lem}

\begin{proof}
The proof follows the same steps as the proof of Lemma~\ref{lem:continuation} and Lemma~\ref{lem:domain}. See also \cite[Lemma~5]{franceschi_2019} to show the inclusion of $\mcG_1$ in the set defined in~\eqref{eq:setBrownian}.
\end{proof}

\begin{lem}[Poles of $F_1$]
\label{lem:poleBrownian}
$F_1$ has one or two poles in the set defined in~\eqref{eq:setBrownian}:
\begin{itemize}
\item
$0$ is always a simple pole of $F_1$,

\item
$s_1^p$ is a simple pole of $F_1$ if and only if $\psi_1(x^-,S_2^\pm(x^-))<0$, where $s_1^p$ is defined in~\eqref{eq:s1pBrown}.
\end{itemize}
$F_2$ has a unique simple pole which is $s_2^p$ if $\psi_2(S_1^\pm(y^-),y^-)<0$ and has no poles otherwise.
\end{lem}

\begin{proof}
The proof follows the same steps (but simpler) as the proof of Lemma~\ref{lem:poleD}. The poles come from the zeros of the denominator of the continuation formula~\eqref{eq:continuationBrown}, that is the zeros of $\psi_1(s_1,S_2^+(s_1))$. It is the intersection between a line and an ellipse, see Figure~\ref{fig:kernelBrownian}.
\end{proof}

\begin{prop}[Asymptotics of domination]
\label{prop:asymptBrown}
The asymptotic behaviour of $1-p_1(u,0)$ as $u\to\infty$ is given by
\[
1-p_1(u,0)\sim C
\begin{cases}
e^{us_1^p} & \text{ if } \psi_1(x^-,S_2^\pm(x^-))<0,\\
u^{-\frac{3}{2}}e^{ux^-} & \text{ if } \psi_1(x^-,S_2^\pm(x^-))>0,\\
u^{-\frac{1}{2}}e^{ux^-} & \text{ if } \psi_1(x^-,S_2^\pm(x^-))=0,
\end{cases}
\]
for some constant $C$ which depends on the case, where $s_1^p$ is defined in~\eqref{eq:s1pBrown}.
\end{prop}

\begin{proof}
The singularities (poles and branch points) of $F_1$ are known from Lemma~\ref{lem:poleBrownian} and equation~\eqref{eq:continuationBrown}. The asymptotics derives from standard transfer theorems as in the proof of Lemma~\ref{prop:asympt}.
\end{proof}

\subsection{Boundary value problem and its solution}
\label{sec:BVPBrownian}
We state an homogeneous Carleman BVP satisfied by the function $f_1$ defined in~\eqref{eq:f1BM}.

\begin{prop}[BVP]
\label{prop:BVPBrownian}
The function $f_1$ satisfies the following Carleman boundary value problem:
\begin{enumerate}[label=(\roman*)]
\item
\label{item:1a}
$f_1(s_1)$ is analytic on $\mcG_1$;

\item
\label{item:2a}
$\lim_{s_1\to\infty} f_1(s_1)=0$;

\item
\label{item:3a}
$f_1$ satisfies the boundary condition on the hyperbola
\[
f_1(\overline{s_1})=G(s_1)f_1(s_1),\quad \forall\, s_1\in\mcH_1,
\]
where
\begin{equation}
\label{eq:Gdefbrownien}
G(s_1)\coloneqq\dfrac{\psi_1}{\psi_2}(s_1,S_2^+(s_1)) \dfrac{\psi_2}{\psi_1}(\overline{s_1},S_2^+(s_1)).
\end{equation}
\end{enumerate}
\end{prop}

\begin{proof}
The proof follows the same steps as that of Proposition~\ref{prop:BVP}.
\end{proof}

Following~\cite{franceschi_tuttes_2016,franceschi_2019} we are going to define the conformal gluing function which glues together the upper part of the hyperbola and its lower part. To that purpose we define for $a\ges 0$ the generalized Chebyshev polynomial for $x\in\mbC\setminus (-\infty,-1]$ by
\[
T_a(x)\coloneqq\cos(a\arccos(x))=\dfrac{1}{2}\left((x+\sqrt{x^2-1})^a+ (x-\sqrt{x^2-1})^a\right).
\]
Let also define the angle of the model
\[
\beta\coloneqq\arccos(-\rho).
\]

\begin{lem}[Conformal gluing function $W$]
\label{lem:gluingwBrownien}
The function
\begin{equation}
\label{eq:wdefBrown}
W(s_1)\coloneqq T_{\frac{\pi}{\beta}}\left(\dfrac{2s_1-(x^++x^-)}{x^+-x^-}\right)
\end{equation}
satisfies the following properties
\begin{enumerate}[label=(\roman*)]
\item
$W$ is holomorphic in $\mcG_1$ and continuous on $\overline{\mcG_1}$;

\item
$W$ is injective in $\mcG_1$;

\item
$W$ satisfies the boundary property
\[
W(s_1)=W(\overline{s_1}),\quad \forall\, s_1\in\mcH_1.
\]
\end{enumerate}
\end{lem}

\begin{proof}
This function has already been studied in several paper. See, for example, \cite[Lemma~3.4]{franceschi_tuttes_2016} or also~\cite[Figure~3]{foschini_equilibria_82} in the case of symmetric conditions.
\end{proof}

To state the main theorem of this section we define
\begin{equation*}
\kappa_1\coloneqq
\begin{cases}
1 & \text{if } 0>S_1^\pm(y^+),\\
0 & \text{if } 0\les S_1^\pm(y^+),
\end{cases}
\quad
\text{and}
\quad
\kappa_2\coloneqq
\begin{cases}
1 & \text{if } \psi_1(x^-,S_2^\pm(x^-))<0 \text{ and } s_1^p>S_1^\pm(y^+),\\
0 & \text{otherwise.}
\end{cases}
\end{equation*}
Using Lemma~\ref{lem:poleBrownian} we note that $\kappa_1$ is defined so that $\kappa_1=1$ when the pole $0$ of $F_1$ is in $\mcG_1$, and $\kappa_1=0$ otherwise. In the same way $\kappa_2=1$ when $s_1^p$ is a pole and is in $\mcG_1$, and $\kappa_2=0$ otherwise.

Let us recall that $W$ is defined in~\eqref{eq:wdefBrown}, $G$ in~\eqref{eq:Gdefbrownien}, $\mcH_1^-$ in Lemma~\ref{lem:hyperbole} and $c$ in~\eqref{eq:defc}.

\begin{thm}[Explicit expression for $F_1$]
\label{thm:explicitF1Brown}
The Laplace transform $F_1$ is given by
\begin{equation}
\label{eq:explicitexpressionBrownian}
F_1(s_1)=p_1(0,0)\left(\dfrac{1}{s_1}+CX(s_1)\right),\quad s_1\in\mcG_1,
\end{equation}
where
\begin{multline}
X(s_1)\coloneqq\left(\dfrac{1}{W(s_1)-W(0)}\right)^{\kappa_1} \left(\dfrac{1}{W(s_1)-W(s_1^p)}\right)^{\kappa_2}\times\\
\times\exp\left(\dfrac{1}{2i\pi} \int_{\mcH_1^-} \log(G(t))\dfrac{W'(t)}{W(t)-W(s_1)}\D t\right),
\end{multline}
\begin{equation}
\label{eq:defConstantC}
C\coloneqq -\dfrac{1}{X(x_0)}\left(\dfrac{1}{x_0}+\dfrac{c}{\psi_1(x_0,0)}\right).
\end{equation}
Furthermore, $p_1(0,0)$ is given by Corollary~\ref{cor:BMp00} for $\rho=0$, whereas for $\rho\in\big(0,\frac{1}{2}\frac{\sigma_2\mu_1}{\sigma_1\mu_2}\big)$ we have
\begin{equation}
\label{eq:p1Brownian}
p_1(0,0)=\dfrac{\frac{1}{2}\sigma^2_2(r_2-\mu_1/\mu_2)\psi_2(S_1^+(y_0),y_0)} {c(\psi_2(S_1^+(y_0),y_0)-\psi_2(0,y_0))-\psi_2(0,y_0)\psi_1(S_1^+(y_0),y_0) \left(1/S_1^+(y_0)+CX(S_1^+(y_0))\right)}.
\end{equation}
and for $\rho\in\big[\frac{1}{2}\frac{\sigma_2\mu_1}{\sigma_1\mu_2},1\big)$ we have
\begin{equation}
\label{eq:p1Brownianbis}
p_1(0,0)=\dfrac{1}{1+C\lim_{s_1\to 0} s_1X(s_1)}.
\end{equation}
where
\begin{multline}
\label{eq:limsXs}
\lim_{s_1\to 0} s_1X(s_1)=\dfrac{1}{W'(0)}\left(\dfrac{1}{W(0)-W(s_1^p)}\right)^{\kappa_2}\times\\
\times\exp\left(\dfrac{1}{2i\pi}\int_{\mcH_1^-} \log(G(t))\dfrac{W'(t)}{W(t)-W(0)}\D t\right).
\end{multline}
\end{thm}

\begin{proof}
The proof follows the same steps as the one of Theorem~\ref{thm:BVPsolution} and also the one of~\cite[Theorem~1]{franceschi_2019}. Solving the BVP of Proposition~\ref{prop:BVPBrownian} in a standard way we find that there exists a constant $C'$ such that
\[
F_1(s_1)=\dfrac{p_1(0,0)}{s_1}+C'X(s_1).
\]
We now compute the value of $C'$. Taking the limit of the kernel equation in $(x_0,0)$ (as in the proof of Lemma~\ref{lem:F0}) we obtain that
\[
0=\psi_1(x_0,0)F_1(x_0)+cp_1(0,0).
\]
Combining this equation with the fact that $F_1(x_0)=\frac{p_1(0,0)}{x_0}+C'X(x_0)$, we deduce that $C'=Cp_1(0,0)$, where $C$ is defined in~\eqref{eq:defConstantC} and we obtain~\eqref{eq:explicitexpressionBrownian}.

It remains to find $p_1(0,0)$ in the case $\rho\in (0,1)$. First, it is important to note that $S_1^+(y_0)\in\mcG_1\cap [0,\infty)$. The positivity is easy to see because
\[
S_1^+(y_0)=\dfrac{2\mu_2\rho\sigma_1/\sigma_2-\mu_1+ \sqrt{(2\mu_2\rho\sigma_1/\sigma_2-\mu_1)^2}}{\sigma_1^2}\ges 0,
\]
and $S_1^+(y_0)\in\mcG_1$, because
\[
S_1^+(y_0)-S_1^+(y_+)=\dfrac{\rho\sigma_1\sigma_2(y^+-y_0)+ \sqrt{(\mu_1-2\mu_2\rho\sigma_1/\sigma_2)^2}}{\sigma_1^2}\ges 0
\]
as $y^+-y_0\ges 0$. We see that $S_1^+(y_0)=0$ if and only if $\rho\ges\frac{1}{2} \frac{\sigma_2\mu_1}{\sigma_1\mu_2}$.

First assume that $S_1^+(y_0)=0$. We obtain with~\eqref{eq:explicitexpressionBrownian}
\[
1=\lim_{s_1\to 0} s_1F_1(s_1)=p_1(0,0)\left(1+C\lim_{s_1\to 0} s_1X(s_1)\right),
\]
which gives~\eqref{eq:p1Brownianbis}. In this case $\kappa_1=1$ and we obtain~\eqref{eq:limsXs}.

Assume now that $S_1^+(y_0)>0$. As in the proof of Corollary~\ref{cor:BMp00} we evaluate the kernel equation at $(0+,y_0)$. We get the same~\eqref{eq:BM2}, even though initially there is the term $\rho\sigma_1\sigma_2$ on both sides. The second equation is obtained by using the point $(S_1^+(y_0),y_0)$:
\[
0=\psi_1(S_1^+(y_0),y_0)F_1(S_1^+(y_0))+\psi_2(S_1^+(y_0),y_0)F_2(y_0)+cp_1(0,0).
\]
The third equation we need is~\eqref{eq:explicitexpressionBrownian} with $s_1=S_1^+(y_0)$:
\[
F_1(S_1^+(y_0))=p_1(0,0)\left(\dfrac{1}{S_1^+(y_0)}+CX(S_1^+(y_0))\right).
\]

Solving these three linear equations with the three unknowns $p_1(0,0)$, $F_2(y_0)$ and $F_1(S_1^+(y_0))$ we obtain~\eqref{eq:p1Brownian}.
\end{proof}

\section{Numerical illustration}
\label{sec:num}
This section provides numerical illustrations of some of our basic formulas. That is, we consider $p_1(0,0)$, the probability of domination by the first component when starting at the origin, for both (i) the Poisson model, see~\eqref{eq:p1Poisson}, and (ii) the Brownian model, see~\eqref{eq:p1Brownian}. The computations were performed using Mathematica and the R~language. It must be mentioned that numerical evaluation of the involved contour integrals is not a straightforward task, and certain care should be taken with the branches of the complex logarithm and the square root.

\begin{figure}[ht!]
\begin{center}
\begin{subfigure}{0.45\textwidth}
\begin{center}
\includegraphics[height=0.18\textheight,width=0.25\textheight]{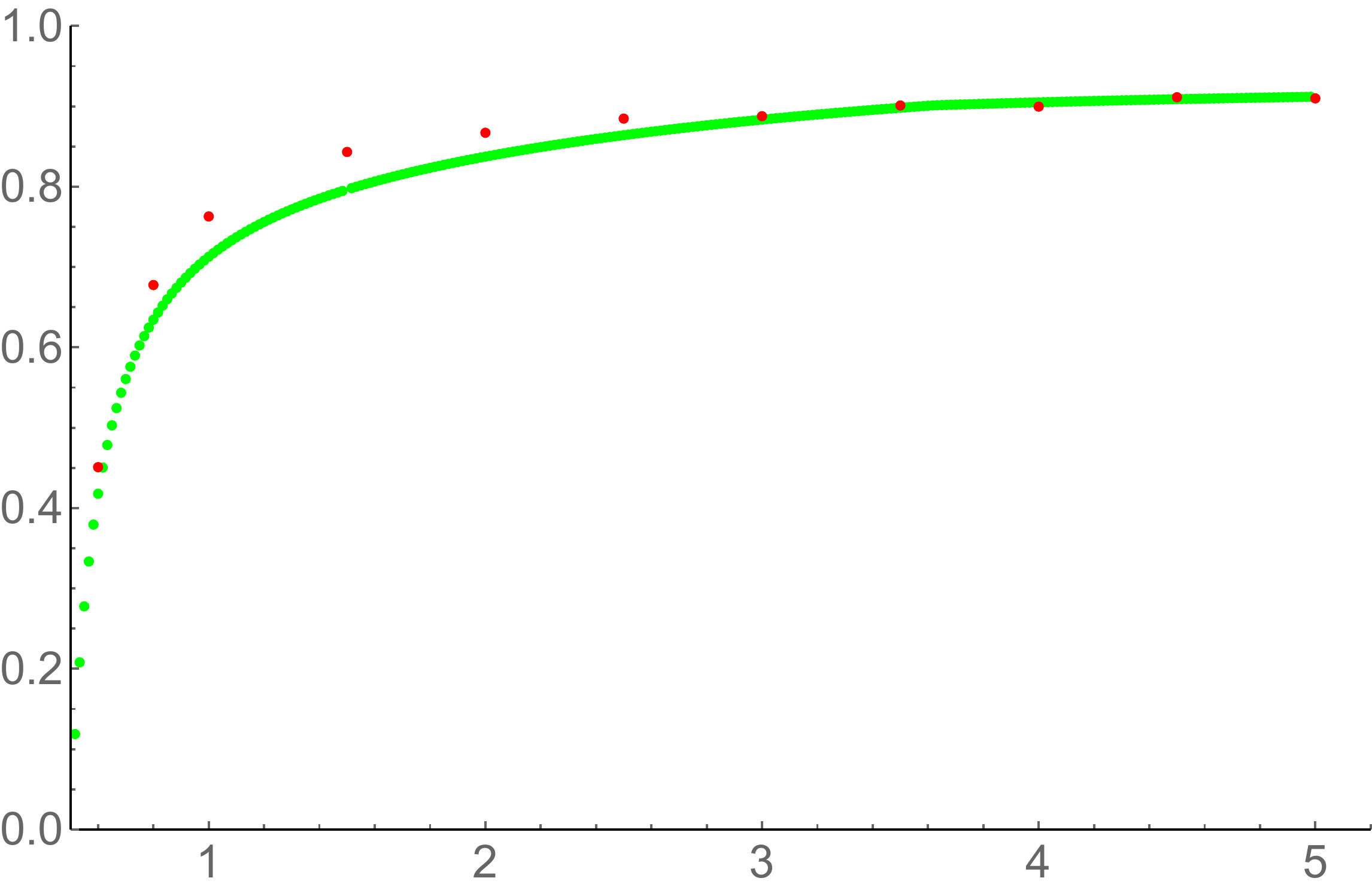}
\caption{Poisson model}
\label{fig:simulations_CPP}
\end{center}
\end{subfigure}
\begin{subfigure}{0.45\textwidth}
\begin{center}
\includegraphics[height=0.18\textheight,width=0.25\textheight]{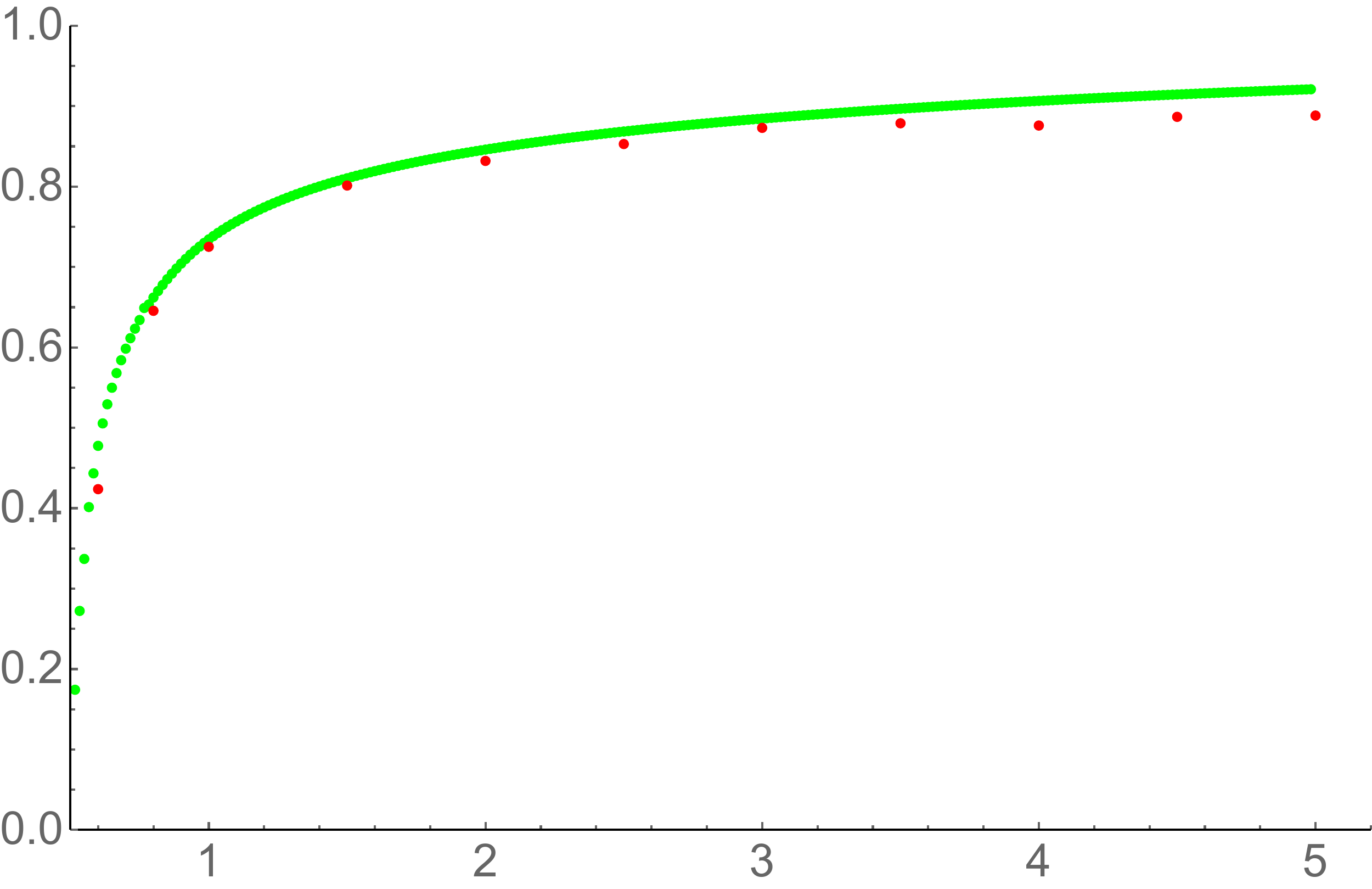}
\caption{Brownian model}
\label{fig:simulations_BM}
\end{center}
\end{subfigure}
\caption{The values of $p_1(0,0)$ computed using contour integrals in green, see~\eqref{eq:p1Poisson} and~\eqref{eq:p1Brownian}, and Monte Carlo simulations in red for a range of $r_2>0.5$.}
\label{fig:numerics}
\end{center}
\end{figure}

Figure~\ref{fig:numerics} presents the plots of $p_1(0,0)$ (in green) as a function of the reflection parameter $r_2>0.5$. For both models we take $r_1=2.5$ and $X_1(1)$, $X_2(1)$ with unit variances and the means $\mu_1=-1$, $\mu_2=-2$. More precisely, in the Poisson model we take $c_1=c_2=1$, $\lambda_1=8$, $\lambda_2=18$, $q_1=4$, $q_2=6$. In the Brownian model we take correlation $\rho=0.2$. It must be mentioned that we use~\eqref{eq:p1Brownian} and not~\eqref{eq:p1Brownianbis}, since $\rho<1/4$. Furthermore, the rates in the Poisson model are rather high, which suggest that the respective uncorrelated Brownian approximation should be close, see~\S\ref{sec:example}. In fact, the corresponding curve drawn basing on the explicit expression in~\eqref{eq:BMp00} almost coincides with the green curve in Figure~\ref{fig:numerics}(a).

In order to check our numerical results, we also perform the Monte Carlo simulation (red dots). It should be stressed that our simulation involves various sources of errors. Firstly, a single run is terminated when $Y_1>100$ and $Y_2/Y_1<0.1$ (at the time of a jump) or the analogous condition is satisfied with the indices swapped. In the first/second case we assume that the first/second component dominates. The Poisson simulation is otherwise exact, whereas the Brownian model is discretized with time-step~$0.01$ so that we reflect a random walk with the corresponding normal increments. In this regard, it is noted that an approximation result similar to that in~\S\ref{sec:approx} can be established also for random walks. Finally, each value is obtained from $10,000$ independent realizations, and thus the $95\%$ asymptotic confidence interval corresponds to $\pm 0.02\sqrt{p_1(1-p_1)}$.

%If your paper includes appendices, then precede the first of them by the command
\appendix
%and then carry on using the \section and \subsection commands, as above.

\section{Derivation of the kernel equation for the Poisson model with common jumps}
\label{sec:derivation}
\begin{proof}[Proof of Proposition~\ref{kernel_equation_Poissonian_approximation}]
For the sake of brevity, here we consider only the case when $r_1\sigma_1>\sigma_2$ and $r_2\sigma_2>\sigma_1$; the derivation of the kernel equations for other cases is similar (the cases with equalities should be considered separately or treated by approximation).

Let $A$ denote the event that the first coordinate dominates the second one. Obviously, for any $u,v>0$ and $h>0$ we have
\begin{equation*}
\begin{gathered}
\begin{aligned}
p_1(u,v)=\p_{(u,v)}(A)&=\p_{(u,v)}(A\cap\{\text{$X$ makes at most one jump on $[0,h]$}\})+\\
&+\p_{(u,v)}(A\cap\{\text{$X$ makes at least two jumps on $[0,h]$}\}).
\end{aligned}
\end{gathered}
\end{equation*}
It is easy to see that the second term is $\Oh(h^2)=\oh(h)$ as $h\to 0+$, uniformly in $(u,v)$.

Now fix arbitrary $u,v>0$. Using the Markov property and considering all possible cases with at most one jump on the time interval $[0,h]$, we obtain
\begin{align*}
p_1(u,v)&=(1-\lambda h)(1-\lambda_1h)(1-\lambda_2h)p_1(u+c_1h,v+c_2h)+\\
&+\lambda h(1-\lambda_1h)(1-\lambda_2h)\int\limits_0^{(\oq_1u)\wedge (\oq_2v)} \D x\; p_1(u-x/\oq_1,v-x/\oq_2)\cdot e^{-x}+\\
&+\lambda h(1-\lambda_1h)(1-\lambda_2h)\int\limits_{\oq_2v}^{\oq_1u}\D x\; p_1(u-x/\oq_1+r_2(x/\oq_2-v),0)\cdot e^{-x}\cdot\indf\left\lbrace\oq_1u> \oq_2v\right\rbrace+\\
&+\lambda h(1-\lambda_1h)(1-\lambda_2h)\int\limits_{\oq_1u}^{\oq_2v}\D x\; p_1(0,v-x/\oq_2+r_1(x/\oq_1-u))\cdot e^{-x}\cdot\indf\left\lbrace\oq_2v> \oq_1u\right\rbrace+\\
&+\lambda h(1-\lambda_1h)(1-\lambda_2h)\int\limits_{\oq_1u}^\infty\D x\; p_1(u-x/\oq_1+r_2(x/\oq_2-v),0)\cdot e^{-x}\times\\
&\times\indf\left\lbrace r_1u-v>(r_1/\oq_1-1/\oq_2)x,\; \oq_1u>\oq_2v\right\rbrace+\\
&+\lambda h(1-\lambda_1h)(1-\lambda_2h)\int\limits_{\oq_2v}^\infty\D x\; p_1(0,v-x/\oq_2+r_1(x/\oq_1-u))\cdot e^{-x}\times\\
&\times\indf\left\lbrace r_2v-u>(r_2/\oq_2-1/\oq_1)x,\; \oq_2v>\oq_1u\right\rbrace+\\
&+\lambda h(1-\lambda_1h)(1-\lambda_2h)\int\limits_{\oq_1u}^\infty\D x\; p_1(0,0)\cdot e^{-x}\cdot\indf\left\lbrace (r_1/\oq_1-1/\oq_2)x\ges r_1u-v,\; \oq_1u>\oq_2v\right\rbrace+\\
&+\lambda h(1-\lambda_1h)(1-\lambda_2h)\int\limits_{\oq_2v}^\infty\D x\; p_1(0,0)\cdot e^{-x}\cdot\indf\left\lbrace (r_2/\oq_2-1/\oq_1)x\ges r_2v-u,\; \oq_2v>\oq_1u\right\rbrace+\\
&+\lambda_1h(1-\lambda h)(1-\lambda_2h)\int\limits_0^{q_1u}\D x\; p_1(u-x/q_1,v)\cdot e^{-x}+\\
&+\lambda_1h(1-\lambda h)(1-\lambda_2h)\int\limits_{q_1u}^\infty\D x\; p_1(0,v+r_1(x/q_1-u))\cdot e^{-x}+\\
&+\lambda_2h(1-\lambda_1h)(1-\lambda h)\int\limits_0^{q_2v}\D x\; p_1(u,v-y/q_2)\cdot e^{-x}+\\
&+\lambda_2h(1-\lambda_1h)(1-\lambda h)\int\limits_{q_2v}^\infty\D x\; p_1(u+r_2(x/q_2-v),0)\cdot e^{-x}+\\
&+\oh(h),\quad h\to 0+.
\end{align*}

Multiplying both sides by $e^{-s_1u-s_2v}$ and integrating the result over $[0,\infty)\times [0,\infty)$ with respect to the variables $u$ and $v$, we obtain
\begin{gather*}
\int\limits_0^\infty \int\limits_0^\infty p_1(u,v)\cdot e^{-s_1u-s_2v}\D u\D v= (1-(\lambda+\lambda_1+\lambda_2)h)\int\limits_0^\infty \int\limits_0^\infty p_1(u+c_1h,v+c_2h)\cdot e^{-s_1u-s_2v}\D u\D v+\\
+\lambda h(I_1+I_2+I_3+I_4+I_5+I_6+I_7)+\lambda_1h(I_8+I_9)+ \lambda_2h(I_{10}+I_{11})+\oh(h),\quad h\to 0+.
\end{gather*}

Noting that
\begin{gather*}
\int\limits_0^\infty \int\limits_0^\infty p_1(u,v)\cdot e^{-s_1u-s_2v}\D u\D v-(1-(\lambda+\lambda_1+\lambda_2)h)\int\limits_0^\infty \int\limits_0^\infty p_1(u+c_1h,v+c_2h)\cdot e^{-s_1u-s_2v}\D u\D v=\\
=\int\limits_{c_1h}^\infty \int\limits_{c_2h}^\infty p_1(u,v)\cdot e^{-s_1u-s_2v}\D u\D v+ \int\limits_{c_1h}^\infty \int\limits_0^{c_2h} p_1(u,v)\cdot e^{-s_1u-s_2v}\D u\D v+\\
+\int\limits_0^{c_1h} \int\limits_{c_2h}^\infty p_1(u,v)\cdot e^{-s_1u-s_2v}\D u\D v+ \int\limits_0^{c_1h} \int\limits_0^{c_2h} p_1(u,v)\cdot e^{-s_1u-s_2v}\D u\D v-\\
-e^{s_1c_1h+s_2c_2h}\int\limits_{c_1h}^\infty \int\limits_{c_2h}^\infty p_1(u,v)\cdot e^{-s_1u-s_2v}\D u\D v+\\
+(\lambda+\lambda_1+\lambda_2)he^{s_1c_1h+s_2c_2h}\int\limits_{c_1h}^\infty \int\limits_{c_2h}^\infty p_1(u,v)\cdot e^{-s_1u-s_2v}\D u\D v=\\
=\left[(\lambda+\lambda_1+\lambda_2-s_1c_1-s_2c_2)F(s_1,s_2)+c_2F_1(s_1)+c_1F_2(s_2)\right]h+\oh(h),
\end{gather*}
we conclude that
\begin{equation}
\label{eq:main_equation}
\begin{split}
(\lambda+\lambda_1+\lambda_2-s_1c_1-s_2c_2)F(s_1,s_2)+c_2F_1(s_1)+c_1F_2(s_2)=\\
=\lambda(I_1+I_2+I_3+I_4+I_5+I_6+I_7)+\lambda_1(I_8+I_9)+\lambda_2(I_{10}+I_{11}).
\end{split}
\end{equation}
To compute $I_i$, $i=1,\ldots,11$, we will use multiple times Fubini's theorem and suitable changes of variables without mention.

For $I_1$ we have
\begin{gather*}
I_1=\int\limits_0^\infty \D u \int\limits_0^\infty \D v \int\limits_0^{(\oq_1u)\wedge (\oq_2v)} \D x\; p_1(u-x/\oq_1,v-x/\oq_2)\cdot e^{-x-s_1u-s_2v}=\\
=\int\limits_0^\infty \D u \int\limits_0^{\oq_1u/\oq_2} \D v \int\limits_0^{\oq_2v} \D x\; p_1(u-x/\oq_1,v-x/\oq_2)\cdot e^{-x-s_1u-s_2v}+\\
+\int\limits_0^\infty \D u \int\limits_{\oq_1u/\oq_2}^\infty \D v \int\limits_0^{\oq_1u} \D x\; p_1(u-x/\oq_1,v-x/\oq_2)\cdot e^{-x-s_1u-s_2v}\eqqcolon I'_1+I''_1.
\end{gather*}
However,
\begin{gather*}
I'_1=\int\limits_0^\infty \D u \int\limits_0^{\oq_1u/\oq_2} \D v \int\limits_0^{\oq_2v} \D x\; p_1(u-x/\oq_1,v-x/\oq_2)\cdot e^{-x-s_1u-s_2v}=\\
=\oq_2\int\limits_0^\infty \D u \int\limits_0^{\oq_1u/\oq_2} \D v \int\limits_0^v \D z\; p_1(u-\oq_2(v-z)/\oq_1,z)\cdot e^{-\oq_2(v-z)-s_1u-s_2v}=\\
=\oq_2\int\limits_0^\infty \D v \int\limits_0^v \D z \int\limits_{\oq_2v/\oq_1}^\infty \D u\; p_1(u-\oq_2(v-z)/\oq_1,z)\cdot e^{-\oq_2(v-z)-s_1u-s_2v}=\\
=\oq_2\int\limits_0^\infty \D v \int\limits_0^v \D z \int\limits_{\oq_2z/\oq_1}^\infty \D y\; p_1(y,z)\cdot e^{-\oq_2(v-z)-s_1(y+\oq_2(v-z)/\oq_1)-s_2v}=\\
=\oq_2\int\limits_0^\infty \D z \int\limits_{\oq_2z/\oq_1}^\infty \D y \int\limits_z^\infty \D v\; p_1(y,z)\cdot e^{-(1+s_1/\oq_1+s_2/\oq_2)\oq_2v+(1+s_1/\oq_1)\oq_2z-s_1y}=\\
=\dfrac{1}{1+s_1/\oq_1+s_2/\oq_2}\int\limits_0^\infty \D z \int\limits_{\oq_2z/\oq_1}^\infty \D y\; p_1(y,z)\cdot e^{-s_1y-s_2z}.
\end{gather*}
Similarly, we have
\[
I''_1=\dfrac{1}{1+s_1/\oq_1+s_2/\oq_2}\int\limits_0^\infty \D z \int\limits_0^{\oq_2z/\oq_1} \D y\; p_1(y,z)\cdot e^{-s_1y-s_2z},
\]
and so
\[
I_1=\dfrac{1}{1+s_1/\oq_1+s_2/\oq_2}\cdot F(s_1,s_2).
\]

For $I_2$ we have
\begin{gather*}
I_2=\int\limits_0^\infty \D u \int\limits_0^\infty \D v \int\limits_{\oq_2v}^{\oq_1u} \D x\; p_1(u-x/\oq_1+r_2(x/\oq_2-v),0)\cdot e^{-x-s_1u-s_2v}=\\
=\int\limits_0^\infty \D v \int\limits_{\oq_2v}^\infty \D x \int\limits_{r_2(x/\oq_2-v)}^\infty \D y\; p_1(y,0)\cdot e^{-x-s_2v-s_1(y+x/\oq_1-r_2(x/\oq_2-v))}=\\
=\dfrac{\oq_2}{r_2}\int\limits_0^\infty \D v\int\limits_0^\infty \D z\int\limits_z^\infty \D y\; p_1(y,0)\cdot e^{-(z+r_2v)/(r_2/\oq_2)-s_2v-s_1(y+\oq_2(z+r_2v)/(r_2\oq_1)-z)}=\\
=\dfrac{1}{r_2(1+s_1/\oq_1+s_2/\oq_2)}\int\limits_0^\infty \D y\; p_1(y,0)\cdot e^{-s_1y}\cdot \int\limits_0^y e^{-(\oq_2/r_2+s_1\oq_2/(r_2\oq_1)-s_1)z}\D z=\\
=\dfrac{1/\oq_2}{(1+s_1/\oq_1+s_2/\oq_2)(1+s_1/\oq_1-s_1r_2/\oq_2)} \left[F_1(s_1)-F_1\left(\dfrac{1+s_1/\oq_1}{r_2/\oq_2}\right)\right],
\end{gather*}
and similarly
\[
I_3=\dfrac{1/\oq_1}{(1+s_1/\oq_1+s_2/\oq_2)(1+s_2/\oq_2-s_2r_1/\oq_1)} \left[F_2(s_2)-F_2\left(\dfrac{1+s_2/\oq_2}{r_1/\oq_1}\right)\right].
\]

Also, we have
\begin{gather*}
I_4=\int\limits_0^\infty \D u \int\limits_0^{\oq_1u/\oq_2} \D v \int\limits_{u\oq_1}^{(r_1u-v)/(r_1/\oq_1-1/\oq_2)} \D x\; p_1(u-x/\oq_1+r_2(x/\oq_2-v),0)\cdot e^{-x-s_1u-s_2v}=\\
=\dfrac{1}{r_2/\oq_2-1/\oq_1}\int\limits_0^\infty \D u \int\limits_0^{\oq_1u/\oq_2} \D v \int\limits_{r_2(\oq_1u/\oq_2-v)}^{(r_1r_2-1)(\oq_1u/\oq_2-v)/(\oq_1(r_1/\oq_1-1/\oq_2))} \D y\; p_1(y,0)\times\\
\times e^{-(y-u+r_2v)/(r_2/\oq_2-1/\oq_1)-s_1u-s_2v}=\\
=\dfrac{1}{r_2/\oq_2-1/\oq_1}\int\limits_0^\infty \D u \int\limits_0^{\oq_1u/\oq_2} \D z \int\limits_{r_2z}^{(r_1r_2-1)z/(\oq_1(r_1/\oq_1-1/\oq_2))} \D y\; p_1(y,0)\times\\
\times e^{-(y-u+r_2(\oq_1u/\oq_2-z))/(r_2/\oq_2-1/\oq_1)-s_1u-s_2(\oq_1u/\oq_2-z)}=\\
=\dfrac{1/\oq_1}{(r_2/\oq_2-1/\oq_1)(1+s_1/\oq_1+s_2/\oq_2)}\int\limits_0^\infty \D z\; e^{\oq_2(1-r_2s_1/\oq_2+s_1/\oq_1)z/(\oq_1(r_2/\oq_2-/\oq_1))}\times\\
\times\int\limits_{r_2z}^{(r_1r_2-1)z/(\oq_1(r_1/\oq_1-1/\oq_2))} \D y\; p_1(y,0)\cdot e^{-y/(r_2/\oq_2-1/\oq_1)}=\\
=\dfrac{1/\oq_1}{(r_2/\oq_2-1/\oq_1)(1+s_1/\oq_1+s_2/\oq_2)}\int\limits_0^\infty \D y\; p_1(y,0)\cdot e^{-y/(r_2/\oq_2-1/\oq_1)}\times\\
\times\int\limits_{\oq_1(r_1/\oq_1-1/\oq_2)y/(r_1r_2-1)}^{y/r_2} \D z\; e^{\oq_2(1-r_2s_1/\oq_2+s_1/\oq_1)z/(\oq_1(r_2/\oq_2-1/\oq_1))}=\\
=\dfrac{1/\oq_2}{(1-r_2s_1/\oq_2+s_1/\oq_1)(1+s_1/\oq_1+s_2/\oq_2)} \left[F_1\left(\dfrac{1+s_1/\oq_1}{r_2/\oq_2}\right)- F_1\left(\dfrac{r_1+s_1(r_1/\oq_1-1/\oq_2)}{(r_1r_2-1)/\oq_2}\right)\right],
\end{gather*}
and similarly
\[
I_5=\dfrac{1/\oq_1}{(1-r_1s_2/\oq_1+s_2/\oq_2)(1+s_1/\oq_1+s_2/\oq_2)} \left[F_2\left(\dfrac{1+s_2/\oq_2}{r_1/\oq_1}\right)- F_2\left(\dfrac{r_2+s_2(r_2/\oq_2-1/\oq_1)}{(r_1r_2-1)/\oq_1}\right)\right].
\]

Also, for $I_6$ we have
\begin{gather*}
I_6=\int\limits_0^\infty \D u \int\limits_0^{\oq_1u/\oq_2} \D v \int\limits_{(r_1u-v)/(r_1/\oq_1-1/\oq_2)}^\infty \D x\; e^{-x-s_1u-s_2v}\cdot p_1(0,0)=\\
=\dfrac{(r_1/\oq_1-1/\oq_2)/\oq_2}{(1+s_1/\oq_1+s_2/\oq_2)(r_1+s_1(r_1/\oq_1-1/\oq_2))}\cdot p_1(0,0),
\end{gather*}
and similarly
\begin{gather*}
I_7=\dfrac{(r_2/\oq_2-1/\oq_1)/\oq_1}{(1+s_1/\oq_1+s_2/\oq_2)(r_2+s_1(r_2/\oq_2-1/\oq_1))} \cdot p_1(0,0).
\end{gather*}

For $I_8$ we have
\begin{gather*}
I_8=\int\limits_0^\infty \D u \int\limits_0^\infty \D v \int\limits_0^{q_1u} \D x\; p_1(u-x/q_1,v)\cdot e^{-s_1u-s_2v-x}=\\
=\dfrac{1}{1+s_1/q_1}\int\limits_0^\infty \D y \int\limits_0^\infty \D v\; p_1(y,v)\cdot e^{-s_1y-s_2v}=\dfrac{1}{1+s_1/q_1}\cdot F(s_1,s_2).
\end{gather*}

For $I_9$ we have
\begin{gather*}
I_9=\int\limits_0^\infty \D u \int\limits_0^\infty \D v \int\limits_{q_1u}^\infty \D x\; p_1(0,v+r_1(x/q_1-u))\cdot e^{-x-s_1u-s_2v}=\\
=\dfrac{1}{1+s_1/q_1} \int\limits_0^\infty \D y \int\limits_0^\infty \D v\; p_1(0,v+r_1y)\cdot e^{-s_2v-q_1y}=\\
=\dfrac{1}{r_1(1+s_1/q_1)} \int\limits_0^\infty \D v \int\limits_v^\infty \D z\; p_1(0,z)\cdot e^{-s_2v-q_1(z-v)/r_1}=\\
=\dfrac{1/q_1}{(1+s_1/q_1)(r_1s_2/q_1-1)}\left[F_2(q_1/r_1)-F_2(s_2)\right].
\end{gather*}

Similarly, we have
\[
I_{10}=\dfrac{1}{1+s_2/q_2}\cdot F(s_1,s_2)
\]
and
\[
I_{11}=\dfrac{1/q_2}{(1+s_2/q_2)(r_2s_1/q_2-1)}\left[F_1(q_2/r_2)-F_1(s_1)\right].
\]

Substituting the obtained values of $I_i$, $i=1,\ldots,11$, into~\eqref{eq:main_equation} and multiplying both sides by $-1$ finishes the proof.
\end{proof}

\section*{Acknowledgements}
The authors are grateful to the anonymous referees for a careful reading of the paper and valuable comments that helped to improve the exposition.

This work was initiated during the visit of S.~Franceschi to Aarhus. The authors gratefully acknowledge financial support of Sapere Aude Starting Grant 8049-00021B ``Distributional Robustness in Assessment of Extreme Risk'' from Independent Research Fund Denmark.

\bibliographystyle{apalike}
\bibliography{Bibliography}
\end{document}